\documentclass{article}

\usepackage{blindtext}
\usepackage{graphicx} 
\usepackage{amsmath}
\usepackage{hyperref}
\usepackage{amssymb}
\usepackage{amsthm}
\usepackage{graphicx}
\usepackage{amsmath}
\usepackage{amscd}
\usepackage{amsfonts}
\usepackage{multicol}
\usepackage{multirow}
\usepackage{url}
\usepackage{mathtools}
\usepackage{comment}
\usepackage{tikz}
\usepackage{tikz-cd}
\usetikzlibrary{matrix,decorations.pathreplacing,calc}
\usepackage{pdflscape}
\usepackage{rotating}
\usepackage{tikz}
\usepackage{authblk}
\usepackage{pgfplots}

\usetikzlibrary{arrows}

\usetikzlibrary{automata}

\usetikzlibrary{chains}

\usetikzlibrary{positioning}

\usetikzlibrary{backgrounds}

\usetikzlibrary{calc,decorations.pathreplacing}

\usepackage{amssymb,amsmath}

\pgfplotsset{compat=1.9}

\makeatletter 
\g@addto@macro\@floatboxreset{\centering} 
\makeatother

\makeatother
\pgfkeys{tikz/mymatrixenv/.style={decoration=brace,every left delimiter/.style={xshift=3pt},every right delimiter/.style={xshift=-3pt}}}
\pgfkeys{tikz/mymatrix/.style={matrix of math nodes,left delimiter=[,right delimiter={]},inner sep=2pt,column sep=1em,row sep=0.5em,nodes={inner sep=0pt}}}
\pgfkeys{tikz/mymatrixbrace/.style={decorate,thick}}

\DeclareMathOperator{\diag}{diag}

\newtheorem{teo}{{Theorem}}[section]

\newtheorem{lema}[teo]{ Lemma}
\newtheorem{prop}[teo]{Proposition}
\newtheorem{remark}[teo]{Remark}

\DeclareMathOperator{\GL}{GL}

\DeclareMathOperator{\id}{id}
\DeclareMathOperator{\defect}{defect}

\DeclareMathOperator{\im}{im}
\DeclareMathOperator{\supp}{supp}

\DeclareMathOperator{\End}{End}

\DeclareMathOperator{\Rep}{Rep}

\DeclareMathOperator{\M}{Mat}
\DeclareMathOperator{\Hom}{Hom}

\DeclareMathOperator{\cogrado}{codeg}
\DeclareMathOperator{\modu}{mod}

\newcommand{\calb}{\ensuremath{\mathcal{B}}}
\newcommand{\calm}{\ensuremath{\mathcal{M}}}

\newcommand{\Ss}{§}

\newcommand{\C}{\mathbb{C}}
\newcommand{\Z}{\mathbb{Z}}
\newcommand{\N}{\mathbb{N}}
\newcommand{\F}{\mathbb{F}}
\newcommand{\A}{\alpha}
\newcommand{\B}{\beta}
\newcommand{\G}{\gamma}
\newcommand{\D}{\delta}
\newcommand{\E}{\varepsilon}

\newcommand{\Si}{\sigma}

\newcommand{\mcm}{\ensuremath{\mathrm{lcm}}}
\newcommand{\mcd}{\ensuremath{\mathrm{gcd}}}

\newcommand{\spec}{\ensuremath{\mathrm{Spec}}}

\newcommand{\lra}{\ensuremath{\rightarrow}}

\newcommand{\parE}{(\!(\E)\!)}
\newcommand{\parEn}{(\!(\E^\frac{1}{n})\!)}

\newcommand{\lp}{(\!(}
\newcommand{\rp}{)\!)}

\title{
A partial classification of simple regular representations of bimodules type $(2,\,2)$ over $\C\parE$}
\author[1]{Hern\'an Giraldo\thanks{hernan.giraldo@udea.edu.co}}
\author[1]{David Reynoso-Mercado\thanks{david.reynoso@udea.edu.co}}
\author[1]{Pedro Rizzo\thanks{pedro.hernandez@udea.edu.co}}
\affil[1]{\'Algebra, Teor\'ia de N\'umeros y Aplicaciones: ERM (ALTENUA) y \'Algebra UdeA, Instituto de M\'atematicas, Facultad de Ciencias Exactas y Naturales, Universidad de Antioquia UdeA, Calle 70 No. 52-21, Medell\'in, Colombia}
\begin{document}

\maketitle
\begin{abstract}

In this paper, we use Galois descent techniques to find suitable representatives of the regular simple representations of the species of type $(2,2)$ over $k_n := k[\varepsilon^{1/n}]$, where $n$ is a positive integer and $k:=\mathbb{C}(\!(\varepsilon)\!)$ is the field of Laurent series over the complexes. These regular representations are essential for the definition of canonical algebras. Our work is inspired by the work done for species of type $(1,4)$ on $k$ in \cite{GR22}. We presents all the regular simple representations on the $n$-crown quiver, and from these, we establish a partial classification of regular simple representations of bimodules type $(2,2)$.\\


\end{abstract}
\section{Introduction}
In \cite{GLS16}, Geiss \textit{et al.} began investigating the representation theory of a new class of quiver algebras associated with symmetrizable generalized Cartan matrices. Their goal was to lay the foundation for generalizing many of the connections between path algebras, preprojective algebras, Lie algebras, and cluster algebras from the symmetric to the symmetrizable case. One of their specific goals was to obtain new geometric constructions of the positive part of a symmetrizable Kac-Moody algebra in terms of varieties of representations of these quiver algebras. For example, they introduced a new class of $\F$-algebras, denoted by $H_\F(C,\,D,\,\Omega)$, on the field $\F$. These algebras allow one to model species representations as locally free modules over them, especially when $\F$ is algebraically closed. Since the algebra $H_\F(C,\,D,\,\Omega)$ is defined via quivers with relations, one can study its module varieties over any field $\F$.  In the particular case that $\F=\C$, they generalize Lusztig's nilpotent varieties from the symmetric to the symmetrizable case.

In this article, we study normal forms for the regular simple representations of species of type $(2,2)$, also denoted as $\tilde{A}_{11}$ in the notation introduced in \cite{DR75}. Our goal is to obtain a partial classification of these representations. From a broader perspective, our partial classification can be seen as a new step forward in the study of species representations of the $\F$-algebra $H_\F(C,\,D,\,\Omega)$. 

Our classification technique, following the ideas in \cite{GR22}, consists of constructing a canonical algebra, as outlined by Ringel in \cite{Ri90}, associated to a bimodule of type $(2,2)$. Specifically, we focus on bimodules with product dimension equal to $4$ over division algebras with center $\C\parE$. There are two distinct families of such bimodules, of type $(1,4)$ and of type $(2,2)$, as discussed in subsection \ref{sec1.3}. The former case has been extensively examined in \cite{GR22}. In this paper, we shift our focus to the latter type and use a special case of representations of the $n-$crown quiver (see subsection \ref{ncrown}) to obtain a classification that differs substantially from that in \cite{GR22}.

Inspired by \cite{GLS20}, we first need to find explicit normal forms for canonical algebras of type $(2,2)$ over the field of complex Laurent series, $\C\parE$. To this end, we must first find normal forms of the regular simple representations of species of type $(2,2)$. Since $\C\parE$ is a quasi-finite field, the technique of Galois descent is an efficient tool for achieving our objectives in this context.

This paper is organized as follows. In Section \ref{Sec:PreyNot}, we recall several concepts and fix some notations that we will use throughout the paper. More precisely, we introduce a brief study of root systems for quivers in Subsection \ref{quivers}. We present a quick tour of string and band modules on string algebras in Subsection \ref{string}. In Subsection \ref{ncrown}, we focus on the classification of indecomposable modules for a path gentle (and therefore string) algebra constructed on the \emph{$n$-crown quiver} $\mathcal{Q}_n$.

In section \ref{main}, we develop the main results of this paper. In fact, in Subsection \ref{isokd} we will establish a decisive isomorphism as $k_d$-algebras between the tensor algebra $k_d\otimes_k\Lambda_n$ and the path algebra $k_d\mathcal{Q}_n$, where $\Lambda_n$ is the \textit{Euclidean species} $\Lambda_n\coloneqq\left[\begin{array}{cc}k_n & k_n\oplus\,^{\Si_n}k_n\\0 & k_n \end{array}\right]$. The purpose of this paper is to classify the $\Lambda_n-$modules of type $(2,\,2)$ over $k_n$. We partially resolve this problem from of the next subsections. Our strategy is as follows. We show that, in Subsection \ref{isokd}, the path algebra $k_{d}\mathcal{Q}_{n}$ admits a $k_n$-linear automorphism of order $d$, which is compatible with the action of $\Si_{d}$ on $k_{d}\otimes\Lambda_n$. We denote this automorphism by $\G_{d,n}$. In Subsection \ref{Invariantrep}, we describe how $\G_{d,n}$ acts on the representations of the $n$-crown quiver $\mathcal{Q}_n$. Using Theorem \ref{teoiso}, we show how a regular simple representation of $\Lambda_n$ can be regarded as a $\mathcal{Q}_n$-representation. In Subsection \ref{GD}, we prove that there is a bijection between certain isoclasses of regular simple $\Lambda_n$-modules and certain special $\langle\Si_d\rangle$-orbits of regular simple $k_d\mathcal{Q}_n$-modules (see Theorem \ref{lemma3.7}). In Subsection \ref{hnhrep}, we prove that any regular simple representation is isomorphic to a sum of irreducible elements in two disjoint families of invariant modules (see Theorem \ref{classifi}). Finally, we present a total classification of regular simple non-homogeneous representations on $\Lambda_n$ and a partial classification of regular simple homogeneous representations on $\Lambda_n$, for any $n\geq 2$. The classes of regular simple non-homogeneous and homogeneous representations correspond to string and band (indecomposable) modules of the path algebra $k_d\mathcal{Q}_n$, respectively. In Subsection \ref{Seclambda} we conclude our classification with the study of the case $n=1$.

\section{Preliminaries}\label{Sec:PreyNot}

Throughout our work, $k:=\C\parE$ denotes the field of Laurent series over the complex numbers, which is a quasi-finite because $\C$ is an algebraically closed field of characteristic $0$ (see \cite[Chapter XIII, \Ss 2]{S79}). Thus, each division algebra over $k$ is isomorphic to $k_n\coloneqq \C\parEn$, for some $n\in\{1,\,2,\,3,\cdots\}=\Z_{>0}$. In particular, every finite-dimensional division algebra is commutative, and the Galois group of the extension $k_n\big\lvert k$ is the cyclic group $C_n=\langle\Si_n\rangle$ of order $n$. Here, $\Si_n$ acts $k$-linearly on $k_n$ via $\Si_n(\E^\frac{j}{n})=\zeta_n^j\E^\frac{j}{n}$, where $\zeta_n\coloneqq e^\frac{2\pi i}{n}$.

For $m$ a divisor of $n$, we identify $k_m$ with the subfield of $k_n$ generated by $(\E^\frac{1}{n})^\frac{n}{m}$, thus $\Si_n|_{k_m}=\Si_m$. We say that $x\in k_n$ is \emph{generic} if $|C_n \cdot x|=n$. This is equivalent to prove that $\Pi_{j=0}^{n-1}(y-\Si_n^j(x))\in k[y]$, which is an irreducible polynomial. For $0\neq x=\sum_{j\in\Z}x_j\E^\frac{j}{n}\in k_n$ we define $\cogrado_n(x)\coloneqq \min\{j\in\Z|x_j\neq0\}$ and $\cogrado(0)=\infty$.

\subsection{Representations of the species of type $(2,\,2)$ over finite-dimensional extensions of $\C\parE$}\label{sec1.3}

In this subsection, we consider bimodules for which the field $k$, acts centrally. The category of these $k_n$-$k_m$-bimodules is equivalent to the category of left $k_n\otimes_kk_m$-modules, where $n$ and $m$ are positive integers.

Let $X$ be a $k_n$-$k_m$-bimodule, and let $\Si_n$ be an automorphism of $k_n$. We define the bimodule $^{\Si_n}X$, where the left multiplication is defined as $y{\ast}x\coloneqq \Si_n(y)x$ for any $x\in X,\, y\in k_n$, and the right multiplication is the usual multiplication as a $k_m$-module. Similarly, we can define $X^{\Si_m}$ as a $k_n$-$k_m$-bimodule.

As we will show in Proposition \ref{obs1}, $k_n\otimes_kk_m$ is isomorphic to $k_{\mcm(n,\, m)}\times\cdots\times k_{\mcm(n,\,m)}$, $\mcd(n,\,m)-$times, as $k$-algebras. Thus, the category of $k_n$-$k_m$-bimodules is semisimple with $\mcd(n,\,m)$-isoclasses of simple objects.

Let $n,\,m\in\Z_{>0}$ and $X$ be a $k_n$-$k_m$-bimodule such that $\dim(\,_{k_n}X)\dim (X_{k_m})=4$. Then it is easy to see that only two cases are possible:
\begin{enumerate}
\item $k_n=k_4,\,k_m =k$ y  $X=k_4$.  
\item $k_n=k_m$ and $X=k_n\oplus \,^{\Si_n}k_n$, for some $n\in \Z_{>0}$. 
\end{enumerate}

The first case has already been addressed in \cite{GR22}, so we will focus on the second case.

Recall that a representation $M$ of $k_n\oplus\,^{\Si_n}k_n$ as a $k_n$-$k_n$-bimodule takes the form $(k_n^m,k_n^{m'},\varphi_M)$, where $\varphi_M:(k_n\oplus,^{\Si_n}k_n)\otimes_{k}k_n^{m'}\lra k_n^{m}$ is a $k_n$-linear morphism, and $m,\,m'\in \Z_{>0}$, as defined in \cite{Ri90}. Let $N$ be the matrix of the mapping $\varphi_M$.

By definition, the objects in the category $\Rep_{k_n}(2,\,2)$ are $m\times 2m'$ matrices with entries in $k_n$. For $N\in \M_{m\times 2m'}(k_n)$ and $N'\in \M_{l\times 2l'}(k_n)$, the morphisms from $N$ to $N'$ are given by the $k_n$-vector space:
$$\Hom_{(2,\,2)}(N,\,N')\coloneqq \{\left(f_2,\,f_1\right)\in\M_{m\times l}(k_n)\times\M_{2l'\times 2m'}(k)|f_2N=N'f_1\}.$$
It is straightforward to see that $\Rep_{k_n}(2,\,2)$ is an abelian $k$-linear category. It is naturally equivalent to $\Lambda_n\mbox{-}\modu$ for the Euclidean species $\Lambda_n\coloneqq\left[\begin{array}{cc}
k_n & k_n\oplus\,^{\Si_n}k_n\\0 & k_n \end{array}\right]$ of type $(2,\,2)$ over $k_n$.

Let $M$ be a representation and $N\in\M_{m\times 2m'}(k_n)$ be the matrix of the morphism
$$\varphi_M:(k_n\oplus\,^{\Si_n}k_n)\otimes_{k}k_n^{m'}\lra k_n^{m}.$$
We denote by $\underline{\dim}N=(m',\,m)\in\N_0^2$. Following Dlab-Ringel \cite[\Ss 3]{DR75}, $N\in\Rep_k(2,\,2)$ is called {\em regular simple} if $\underline{\dim}N=(m,m)$ for some $m\in\Z_{>0}$ and $\End_{(2,\,2)}(N)$ is a division algebra. Since $k_n$ is quasi-finite, the previous condition is equivalent to $\End_{(2,\,2)}(N)\cong k_l$ for some $l$.

\subsection{Quivers and roots}\label{quivers}
Let $Q$ be a finite and connected quiver. As usual, we denote by $Q_0$ (resp. by $Q_1$) the set of vertices (resp. the set of arrows) of $Q$. Also, $h(a)$ (resp. $t(a)$) denotes the vertex of $Q_0$ where the arrow $a$ starts (resp. ends).

A {\it subquiver} of a quiver $Q=(Q_0,\,Q_1,\,h,\,t)$ is a quiver $Q'=(Q'_0,\,Q'_1,\,h',\,t')$ where $Q'_0\subset Q_0$, $Q'_1\subset Q_1$, $h'=h|_{Q'_1}$ and $s'=s|_{Q'_1}$. A subquiver $Q'$ of $Q$ is \textit{full} if $Q'_1=\{a\in Q_1|h(a),\,t(a)\in Q'_0\}$.

Associated to every quiver $Q$ we have a \textit{symmetric generalized Cartan matrix}, $A_Q=(a_{ij})$, whose entries are defined as follows:
\begin{eqnarray}
    a_{ij}=\left\{\begin{array}{lc}
        2 & \textit{if } i=j;\\
        -\#\{\textit{edges between $i$ and $j$}\} & \textit{if } i\neq j. 
    \end{array} \right.
\end{eqnarray}
Note that, $A_Q$ not depend of the orientation on $Q$.

The {\it root lattice} of $Q$ is defined by the free abelian group $\Z^{Q_0}$. We identify each trivial path $e_i$ with the $i-$th standard basis vector of $\Z^{Q_0}$, for each $i\in Q_0$. This group admits a partially order ``$\boldsymbol{\geq}$'' induced, for $a=\sum_{i\in Q_0}a_ie_i\in \Z^{Q_0}$, by the relation:  
$$
a=\sum_{i\in Q_0}a_ie_i \boldsymbol{\geq} \boldsymbol{0}\ \text{if and only if}\ a_i\geq 0\ \text{for all}\ i\in Q_0,
$$
where $\boldsymbol{0}$ is the zero object in $\Z^{Q_0}$. In consequence, $a=\sum_{i\in Q_0}a_ie_i \boldsymbol{\geq} b=\sum_{i\in Q_0}b_ie_i$ if and only if $a-b\boldsymbol{\geq}\boldsymbol{0}$.




From $A_Q$ we can define a \textit{symmetric bilinear form} $(-,-):\Z^{Q_0}\times \Z^{Q_0}\rightarrow \Z$, which satisfies $(e_i,\,e_j):=a_{ij}$, for all $i,j$. This yields a \textit{reflection} $r_i:\Z^{Q_0}\lra \Z^{Q_0}$ \textit{on} $\Z^{Q_0}$, which is defined by $a\mapsto a-(a,e_i)e_i$, for each $i\in Q_0$. We define the \textit{Weyl Group} of $Q$, denoted by $\mathcal{W}$, as the subgroup of $\text{Aut}(\Z^{Q_0})$ generated by the all reflections $r_i$. We now introduce the \textit{set of (positive) imaginary roots for} $Q$, denoted by $\Delta^+_{\im}:=\bigcup_{w\in\mathcal{W}}w(F)$, as the union of the sets of images $w(F)$, where $w$ ranges over all elements of the Weyl group $\mathcal{W}$. Here the set $F$, which is called \textit{the fundamental region}, is defined by
$$
F:=\{a\in \Z^{Q_0}\mid a\neq0,\,(a,e_i)\leq0\ \text{for all}\ i\in Q_0,\ \text{and $\supp a$ is connected}\},
$$
where $\supp a$ represents the full subquiver of $Q$ with vertex set $\{i\in Q_0\mid a_i\neq 0\}$. The quiver $\supp a$ is called \textit{the support of} $a\in \Z^{Q_0}$.

Due to the above comments, we denote \textit{the minimal positive imaginary root} by $\delta$. If $M=(V_i, f_\rho)$ is an indecomposable representation of the quiver $Q$ over field $\F$, then the {\it defect} of $M$ is defined by the integer
\begin{eqnarray*}
    \defect(M):=\langle\delta,\dim M\rangle,
\end{eqnarray*}
where ``$\langle \cdot,\,\cdot\rangle$'' is \textit{the Euler form of} $Q$, which is defined by
\begin{eqnarray*}
\langle a,\,b\rangle:=\sum_{i\in Q_0}a_ib_i-\sum_{\rho:i\rightarrow j}a_ib_j, 
\end{eqnarray*}
and $\dim M:=\sum_{i\in Q_0}(\dim V_i)e_i$. 

Let $Q$ be a connected quiver without oriented cycles. According to the theory above, we have the following important criteria 
for the isomorphism classes of indecomposable representations $M$ of $Q$:
\begin{enumerate}
    \item $M$ is preprojective if and only if $\tau^rM=0$ for some sufficiently large integer $r$ if and only if $\defect(M)<0$.
    \item $M$ is preinjective if and only if $\tau^{-r}M=0$ for some sufficiently large integer $r$ if and only if $\defect(M)>0$.
    \item $M$ is regular if and only if $\tau^{-r}\tau^rM=M$ for all integers $r$ if and only if $\defect(M)=0$.
\end{enumerate}
Here, $\tau$ represents the \textit{Auslander-Reiten translate} (see \cite{ASS06}).

\subsection{String Algebras}\label{string}
Let $\F Q/I$ be a bound quiver algebra. We say that $\F Q/I$ is a {\it string algebra} (see \cite{BR87}) 
if it satisfies the following three conditions:
\begin{enumerate}
    \item[(S1)] Any vertex of $Q$ is starting or ending point of at most two arrows.
    \item[(S2)] Given an arrow $\B\in Q_1$, there exists at most one arrow $\rho\in Q_1$ such that either $\B\rho\notin I$ or $\rho\B\notin I$.
    \item[(S3)] The ideal $I$ is generated by null relations. 
\end{enumerate} 
Let $\F Q/I$ be a string algebra and $a\in Q_1$ be an arrow. The {\it formal inverse of} $a$ is the arrow $a^{-1}$ such that $t(a^{-1})=h(a)$ and $h(a^{-1})=t(a)$. A {\it string} $S$ is a sequence $S=s_1\cdots s_n$, where $s_i\in Q_1$ or $s_i^{-1}\in Q_1$, $t(s_i)=h(s_{i-1})$ and $s_i\neq s_{i-1}^{-1}$ for all $1<i\leq n$. Let $\mathcal{S}$ be the set of all strings. We define an equivalence relation $S\sim_s S'$ if and only if $S=S'$ or $S'=S^{-1}$. We denote the set of representatives of the equivalence classes of $\mathcal{S}/\sim_s$ as $\underline{\mathcal{S}}$. Here $S^{-1}$ corresponds to the string $S^{-1}=s_n^{-1}\cdots s_1^{-1}$.

Let $\F Q/I$ be a string algebra and $S=s_1s_2\cdots s_n\in \underline{\mathcal{S}}$ be a string. We define a map $u:\{0,\,1,\cdots,\,n\}\lra Q_0$ by $u(0)=t(s_1)$ and $u(i)=h(s_i)$, for all $1\leq i\leq n$. We will define a representation $M(S)$ of the quiver $Q$ bounded by the relations $I$ associated to the string $S$. Indeed, for each vertex $v\in Q_0$, let $I_v=u^{–1}(v)\subset\{0,\,1,\cdots,\,n\}$ be the set of indexes. To each vertex $v$ we associate the vector space $M(S)_v=\F^{\# I_v}$, where $\# I_v$ denotes the cardinality of the set $I_v$. Now, from the choice of a basis $\{z_0,\,z_1,\cdots,\,z_n\}$ of $\F^{n+1}$, we take the corresponding elements indexed by $I_v$ to obtain a basis of the vector space $\F^{\# I_v}$, for each $v\in Q_0$. We define, for each arrow $a\in Q_1$ such that $M(S)_{t(a)}\neq 0$ and $M(S)_{h(a)}\neq 0$, a morphism $f_a:M(S)_{t(a)}\lra M(S)_{h(a)}$ by
\begin{eqnarray*}
    f_a(z_i)=\left\{\begin{array}{ll}
         z_{i+1}& \mbox{if } s_{i+1}=a,  \\
         z_{i-1}& \mbox{if } s_{i}=a^{-1},  \\
         0& \mbox{otherwise.}  \\         
         \end{array}\right.
\end{eqnarray*}
It is easy to see that $M(S)=(M(S)_v,\,f_a)_{v\in Q_0,\,a\in Q_1}$ is a representation of the quiver $Q$ bounded by the relations $I$ and $M(S)\cong M(S')$ if and only if $S\sim_s S'$. The class by isomorphism of this modules are called {\it string modules}.

Let $\mathcal{S}'$ be the subset of $\mathcal{S}$ consisting of all non-trivial strings $S$ such that $S^n$ is defined for all $n\in \N$, and $S$ is not a power of a substring. We define an equivalence relation on $\mathcal{S}'$ as follows: $S\sim_r S'$ if and only if $S'$ is a cyclic permutation of $S$. We denote the set of representatives of the equivalence classes of $\mathcal{S}'/\sim_r$ as $\underline{\mathcal{S}'}$. 

Let $Z$ be a vector space and $\varphi:Z\lra Z$ be an automorphism. It is possible to prove that (see \cite{BR87}) the pair $(Z,\varphi)$ induce a $\F[T,T^{-1}]-$module structure over $Z$. We want to define a module $M(S,\,\varphi)$ associated to each string $S$ in this new class. As before, like for string modules, for each vertex $v\in Q_0$ we consider $I_v=u^{-1}(v)$, $I'_v=I_v\cap\{0,\,1,\cdots,\,n-1\}$ and $s_n:u(n-1)\lra u(0)$.

Let $\bigoplus_{i\in I'_v}Z_i$ be the $\F$-vector space, where $Z_i=Z$ for all $i\in I'_v$. Now, for each arrow $a\in Q_1$, we define $f_a:(\bigoplus Z_i)_{i\in I'_{t(a)}}\lra(\bigoplus Z_j)_{j\in I'_{h(a)}}$, whose domain and codomain are not null spaces. Thus, we denote by $f_i$, for all $1\leq i\leq n$, the following morphisms:
\begin{eqnarray*}
    f_i=\left\{\begin{array}{ll}
         \id:Z_{i-1}\lra Z_i& \mbox{if }s_i\in Q_1\mbox{ and } i<n, \\
         \id:Z_i\lra Z_{i-1}& \mbox{if }s^{-1}_i\in Q_1\mbox{ and } i<n, \\
         \varphi:Z_{n-1}\lra Z_0& \mbox{if }s_i\in Q_1\mbox{ and } i=n, \\
         \varphi^{-1}:Z_i\lra Z_{i-1}& \mbox{if }s^{-1}_i\in Q_1\mbox{ and } i=n. \\
    \end{array}\right.
\end{eqnarray*}
Once we have chosen an order for the direct sums $(\bigoplus Z_i)_{i\in I'_{t(a)}}$ and $(\bigoplus Z_j)_{j\in I'_{h(a)}}$ we can define a matrix
\begin{eqnarray*}
    g_{rs}=\left\{\begin{array}{ll}
         f_i&\mbox{if $t(s_i)$ is the $r$-th component of $(\bigoplus Z_i)_{i\in I'_{t(a)}}$}    \\
         & \mbox{and $h(s_i)$ is the $s$-th component of $(\bigoplus Z_j)_{j\in I'_{h(a)}}$,}\\
         0& \mbox{otherwise.}
    \end{array}\right.
\end{eqnarray*}
It is not difficult to verifies that $M(S,\varphi)$ is a representation of the quiver $Q$ bounded by the relations $I$ and $M(S,\varphi)\cong M(S',\varphi')$ if and only if $(Z,\varphi)$ and $(Z,\varphi')$ are isomorphic as $\F[T,T^{-1}]-$modules and $S\sim_r S'$. These modules are called {\it band modules}.

The main result related to the classification of indecomposable modules for string algebras is the following theorem of Butler and Ringel in \cite{BR87}.
\begin{teo}\label{teostringband}
    Let $\F Q/I$ be a string algebra of finite dimension. The string modules $M(S)$ with $S\in\underline{\mathcal{S}}$ and band modules $M(S',\varphi)$ with $S'\in\underline{\mathcal{S'}}$ and $\varphi:Z\lra Z$ be an automorphism on the vector space $Z$, provide a complete list of indecomposable $\F Q/I$-modules, up to isomorphism.
\end{teo}

\subsection{About the $n$-crown quiver}\label{ncrown}

If $\mathcal{Q}=(Q_0,Q_1,h,t)$ is a quiver, then we call a vertex a {\it source} if it is not head of any arrow and it is a {\it sink} if it is not the tail of any arrow. Let $\mathcal{Q}_n$ be a quiver whose set of vertices $Q_0=U\cup W$ is such that $U=\{1,\,2,\cdots,\,n\}$ corresponds to the set of sources and $W=\{1',\,2',\cdots,\,n'\}$ corresponds to the set of sinks. The quiver $\mathcal{Q}_n$ is called a $n-${\it crown} if $n\geq 2$ and:
\begin{enumerate}
    \item[(C1)] There exists an arrow from $i$ to $j'$ if: $i>1$ and $i\leq j'\leq i+1$; or $i=1$ and $j'=1'$; or $i=n$ and $j'=1'$.
    \item[(C2)] If $n = 2$, no path from $1$ to $2'$ shares a vertex with any path starting at $2$ and ending at $1'$.
\end{enumerate}
For instance, the $2-$crown quiver $\mathcal{Q}_2$ is represented by
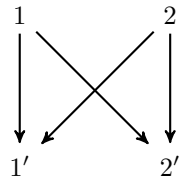
\begin{center}
    \begin{tikzpicture}
[->,>=stealth',shorten >=1pt,auto,node distance=2cm,thick,main node/.style=]
\node (1) at (0,2) {$1$}; 
\node (2) at (0,0) {$1'$}; 
\node (3) at (2,2) {$2$}; 
\node (4) at (2,0) {$2'$}; 
\path[every node/.style={font=\sffamily\small}]     
(1) edge node   {} (2)         
     edge node  {} (4)            
(3) edge node  {} (2)              
   edge node   {} (4);
\end{tikzpicture}
\end{center}

For $n>2$, the $n$-crown quiver $\mathcal{Q}_{n}$ corresponds to the Figure \ref{figurecrown}.
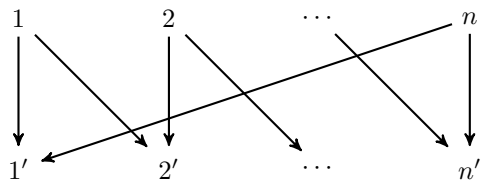
\begin{figure}[h]
\begin{tikzpicture}
[->,>=stealth',shorten >=1pt,auto,node distance=2cm,thick,main node/.style=]
  \node[main node] (1) {$1$};   
  \node[main node] (6) [below of =1]{$1'$};   
  \node[main node] (2) [right of=1] {$2$};
  \node[main node] (7) [below of =2]{$2'$};   
  \node[main node] (3) [right of=2] {$\cdots$};
    \node[main node] (8) [below of =3]{$\cdots$};
  \node[main node] (4) [right of=3] {$n$};
  \node[main node] (5) [below of=4] {$n'$};
\path[every node/.style={font=\sffamily\small}]     
(1) edge node   {} (7)         
     edge node         {} (6)           
(2) edge node  {} (7)              
     edge node         {} (8)
(4) edge node  {} (5)              
   edge node         {} (6)
(3) edge node  {} (5)              ; 
\end{tikzpicture}
\caption{The $n$-crown quiver $\mathcal{Q}_{n}$, for $n\geq 3$.} 
\label{figurecrown}
\end{figure}

If we consider the $n$-crown quiver $\mathcal{Q}_{n}$ and we denote by $\A_i:i\rightarrow i'$ and $\A_{i'}:i-1\rightarrow i'$ the arrows connecting source and sinks, then, it is not difficult to verifies that, the strings with length $j$, where $1\leq j\leq n$, whose starting vertex is $1'$ are defined by
\begin{eqnarray*}
    s_{1',j}:=\left\{\begin{array}{ll}
         \A_1^{-1}\A'_2\A^{-1}_2\cdots\A_{\frac{j+1}{2}}^{-1}&\mbox{ if $j$ is odd,}  \\
         &\\
         \A_1^{-1}\A'_2\A^{-1}_2\cdots\A'_{\frac{j+2}{2}}&\mbox{ if $j$ is even.} 
    \end{array}\right.
\end{eqnarray*}
Similarly, it is not difficult to verifies that, the strings with length $j$, where $1\leq j\leq n$, whose starting vertex is $1$ are defined by
\begin{eqnarray*}
    s_{1,j}:=\left\{\begin{array}{ll}
         \A'_2\A^{-1}_2\cdots\A'_{\frac{j+3}{2}}&\mbox{ if $j$ is odd,}  \\
         &\\
         \A'_2\A^{-1}_2\cdots\A_{\frac{j+2}{2}}^{-1}&\mbox{ if $j$ is even.} 
    \end{array}\right.
\end{eqnarray*}
Now, if we consider $\G_n:\{1,\cdots,\,n\}\lra \{1,\cdots,\,n\}$ a permutation defined by $\G_n(1)=n$ and  $\G_n(i)=i-1$ for all $1<i\leq n$, then, for $1\leq i\leq n$, we have that
\begin{eqnarray*}
    \G_n^{n+1-i}(1)=\G_n^{n-i}(n)=n-(n-i)=i.
\end{eqnarray*}
In consequence, the strings with length $j$, where $1\leq j\leq n$, whose starting vertices are $i'$ and $i$, respectively, are denoted by

\begin{eqnarray*}
    s_{i',j}&:=&\left\{\begin{array}{ll}
         \A_{\G_n^{n+1-i}(1)}^{-1}\A'_{\G_n^{n+1-i}(2)}\cdots\A_{\G_n^{n+1-i}(\frac{j+1}{2})}^{-1}&\mbox{ if $j$ is odd,}  \\
         &\\
         \A_{\G_n^{n+1-i}(1)}^{-1}\A'_{\G_n^{n+1-i}(2)}\cdots\A'_{\G_n^{n+1-i}(\frac{j+2}{2})}&\mbox{ if $j$ is even,} 
    \end{array}\right.
\\
    s_{i,j}&:=&\left\{\begin{array}{ll}
         \A'_{\G_n^{n+1-i}(2)}\A^{-1}_{\G_n^{n+1-i}(2)}\cdots\A'_{\G_n^{n+1-i}(\frac{j+3}{2})}&\mbox{ if $j$ is odd,}  \\
         &\\
         \A'_{\G_n^{n+1-i}(2)}\A^{-1}_{\G_n^{n+1-i}(2)}\cdots\A_{\G_n^{n+1-i}(\frac{j+2}{2})}^{-1}&\mbox{ if $j$ is even.} 
    \end{array}\right.
\end{eqnarray*}

We claim that each vertex appears at most once in a string of length $j$. Indeed, we have:
\begin{enumerate}
    \item For $j$ odd, $v=h(\A)$ or $v=t(\A)$ for some $\A\in s_{i',j}$ if and only if $v\in U_{i,\,\frac{j+1}{2}}\cup W_{i,\,\frac{j+1}{2}}$,
    \item for $j$ even, $v=h(\A)$ or $v=t(\A)$ for some $\A\in s_{i',j}$ if and only if $v\in U_{i,\,\frac{j}{2}}\cup W_{i,\,\frac{j+2}{2}}$,
    \item for $j$ odd, $v=h(\A)$ or $v=t(\A)$ for some $\A\in s_{i,j}$ if and only if $v\in U_{i,\,\frac{j+1}{2}}\cup W_{i,\,\frac{j+3}{2}}-\{i'\}$,
    \item for $j$ even, $v=h(\A)$ or $v=t(\A)$ for some $\A\in s_{i',j}$ if and only if $v\in U_{i,\,\frac{j+2}{2}}\cup W_{i,\,\frac{j+2}{2}}-\{i'\}$,
\end{enumerate}
where the sets of sources and sinks are, respectively, $U_{i,l}:=\left\{i,\,\G_n^{n+1-i}(2),\cdots,\G_n^{n+1-i}\left(l\right)\right\}\subseteq U$ and $W_{i,l}:=\left\{i',\,\G_n^{n+1-i}(2)',\cdots,\G_n^{n+1-i}\left(l\right)'\right\}\subseteq W$, with $1\leq i,\, l\leq n$.

Therefore, for each string $ s_{i',j}$, the corresponding string module $M(s_{i',j})=(M_v, f_a)_{v\in Q_0,a\in Q_1}$ is defined by:
\begin{eqnarray*}
    M_v&=&\left\{\begin{array}{cl}
     k_n &\mbox{if $j$ is odd and $v\in U_{i,\,\frac{j+1}{2}}\cup W_{i,\,\frac{j+1}{2}}$;}  \\
     k_n &\mbox{if $j$ is even and $v\in U_{i,\,\frac{j}{2}}\cup W_{i,\,\frac{j+2}{2}}$;}  \\
     0 & \mbox{otherwise;}
\end{array}\right.\\
    f_a&=&\left\{\begin{array}{cl}
     1& \mbox{if }a\in s_{i',j}\mbox{ or  }a^{-1}\in s_{i',j}; \\
     0 & \mbox{otherwise;}
\end{array}\right.
\end{eqnarray*}
Analogously, for each string $ s_{i,j}$, the corresponding string module $M(s_{i,j})=(M_v, f_a)_{v\in Q_0,a\in Q_1}$ is defined by:
\begin{eqnarray*}
    M_v&=&\left\{\begin{array}{cl}
     k_n &\mbox{if $j$ is odd and $v\in U_{i,\,\frac{j+1}{2}}\cup W_{i,\,\frac{j+3}{2}}-\{i'\}$;}  \\
      k_n &\mbox{if $j$ is even and $v\in U_{i,\,\frac{j+2}{2}}\cup W_{i,\,\frac{j+2}{2}}-\{i'\}$;}  \\
     0 & \mbox{otherwise;}
\end{array}\right.\\
    f_a&=&\left\{\begin{array}{cl}
     1& \mbox{if }a\in s_{i,j}\mbox{ or  }a^{-1}\in s_{i,j}; \\
     0 & \mbox{otherwise.}
\end{array}\right.
\end{eqnarray*}
Thus, under the notations in Section \ref{quivers}, we obtain that
\begin{eqnarray*}
   \dim(M(s_{i',j}))&=&\left\{\begin{array}{ll}
        \sum_{v\in U_{i,\,\frac{j+1}{2}}}e_v+\sum_{v\in W_{i,\,\frac{j+1}{2}}}e_v & \mbox{ if $j$ is odd;} \\
         \sum_{v\in U_{i,\,\frac{j}{2}}}e_v+\sum_{v\in W_{i,\,\frac{j+2}{2}}}e_v & \mbox{ if $j$ is even};
    \end{array}\right.\\
    \dim(M(s_{i,j}))&=&\left\{\begin{array}{ll}
        \sum_{v\in U_{i,\,\frac{j+1}{2}}}e_v+\sum_{v\in W_{i,\,\frac{j+3}{2}}-\{i'\}}e_v & \mbox{ if $j$ is odd;} \\
         \sum_{v\in U_{i,\,\frac{j+2}{2}}}e_v+\sum_{v\in W_{i,\,\frac{j+2}{2}}-\{i'\}}e_v & \mbox{ if $j$ is even.}
    \end{array}\right.
\end{eqnarray*}

Hence, the Cartan matrix of $\mathcal{Q}_n$, indexed by $U\cup W$, is defined by:
$$a_{ij}=\left\{\begin{array}{cl}
     2&\mbox{if }i=j,\\
     -1&\mbox{ if $i\in U$ and j=i' or j=(i+1)',}\\
     -1&\mbox{ if $j\in U$ and i=j' or i=(j+1)',}\\
     0&\mbox{otherwise.}
\end{array}\right.$$

In consequence, it is not difficult to prove that $\delta:=\sum_{i\in U}e_i+\sum_{i'\in W}e_{i'}$ is a minimal (positive) imaginary root and, for $a:=\sum_{i\in U}a_ie_i+\sum_{i\in W}a'_ie_i\in \Z^{Q_0}$, we have that:
\begin{eqnarray}\label{deltadef}
    \langle\D,a\rangle=\sum_{i\in U}a_i+\sum_{i\in W}a'_i-\sum_{i\in U}(a_i'+a'_{i+1}).
\end{eqnarray} 

\begin{lema}\label{lemmastring}
Let $S$ be a string of length $1\leq j\leq n$ in the quiver $\mathcal{Q}_n$ and let $M(S)$ be the corresponding string module. Then,  
\begin{enumerate}
\item $\defect(M(S))=0$ if $j$ is odd,
\item $\defect(M(S))\neq0$ if $j$ is even.
\end{enumerate}  
\end{lema}
\begin{proof}
Let $S=s_1\cdots s_j$ be a string of $\mathcal{Q}_n$ of length $j$, with $1\leq j\leq n$. We denote by 
    $$
    U_S=\{v\in U\mid v=t(s_1)\ \text{or}\ v=h(s_i)\ \text{ for some }\ 1\leq i\leq j\}
    $$
the corresponding set of sources of $S$ and by
$$
 W_S=\{v\in W\mid v=t(s_1)\ \text{or}\ v=h(s_i)\ \text{ for some}\ 1\leq i\leq j\}
$$
the corresponding set of sinks of the string $S$. Thus, 
$$
 \defect(M(S))=\langle\D,\dim(M(S))\rangle=\sum_{i\in U_S}1+\sum_{i\in W_S}1-\sum_{i\in U_S}2=|U_S|+|W_S|-2|U_S|=|W_S|-|U_S|.
$$
In consequence, if either $S=s_{i',j}$ or $S=s_{i,j}$, then it is easy to check that, $|W_S|=|U_S|$, if $j$ is odd, and $|W_S|\neq|U_S|$, if $j$ is even. Therefore, $\defect(M(S))=0$ if $j$ is odd, and  $\defect(M(S))\neq0$ if $j$ is even, which completes the proof of lemma.
\end{proof}

For the case of bands, it is easy to check that there exists a unique band representant in $\underline{\mathcal{S'}}$ in the quiver $\mathcal{Q}_n$. In fact, it is sufficient consider $S'=(\A'_1)^{-1}\A_n(\A'_n)^{-1}\cdots(\A'_2)^{-1}\A'_1$. In this case, if $V$ is a $k_n$ vector space and $\varphi$ is an automorphism of $V$, the corresponding band module is defined by $M(S',\varphi)=(M_v,f_a)_{v\in Q_0,a\in Q_1}$, where $M_v=V$ for each $v\in Q_0$, and the maps 
$$f_a=\left\{\begin{array}{cl}
     \varphi&\mbox{if }a=\A_1,   \\
     \id_V&\mbox{otherwise,} 
\end{array}\right.$$
$M(S',\varphi)$ is represented in the Figure \ref{figureband}.
\begin{figure}
\begin{center}
\begin{tikzpicture}
[->,>=stealth',shorten >=1pt,auto,node distance=2cm,thick,main node/.style=]
  \node[main node] (1) {$V$};   
  \node[main node] (6) [below of =1]{$V$};   
  \node[main node] (2) [right of=1] {$V$};
  \node[main node] (7) [below of =2]{$V$};   
  \node[main node] (3) [right of=2] {$\cdots$};
    \node[main node] (8) [below of =3]{$\cdots$};
  \node[main node] (4) [right of=3] {$V$};
  \node[main node] (5) [below of=4] {$V$};
\path[every node/.style={font=\sffamily\small}]     
(1) edge node   {$\varphi$} (6)         
     edge node         {$\id_V$} (7)           
(2) edge node  {$\id_V$} (7)              
(4) edge node  {$\id_V$} (5)              
   edge node         {$\id_V$} (6)
   (3) edge node  {$\id_V$} (5) ;
\end{tikzpicture}
\end{center}
\caption{Let $V$ be a $k_n$-vector space and $\varphi$ is an automorphism of $V$, the band module $M(S',\varphi)$.} 
\label{figureband}
\end{figure}
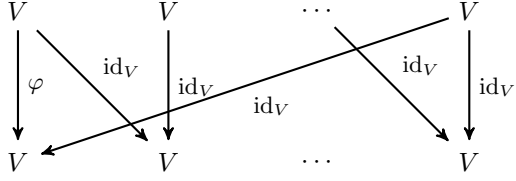

\begin{lema}\label{lemmabands}
Let $S'$ be the unique band, up to isomorphism, in the quiver $\mathcal{Q}_n$ and let $M(S',\varphi)$ be the corresponding band module. Then,
$\defect(M(S',\varphi))=0$.
\end{lema}
\begin{proof}
In this case, we have that $\dim(M(S',\varphi))=\sum_{v\in Q_0}\dim(V)e_v=\dim(V)\delta$. Thus,
$$
\defect(M(S',\varphi))=\langle\delta,\dim(M(S',\varphi))\rangle=\langle\delta,\dim(V)\delta\rangle=\dim(V)\langle\delta,\delta\rangle=0.
$$
\end{proof}

We will see in \ref{isokd} that $k_{d}\otimes_{k_n} \Lambda_n\cong k_{d}\mathcal{Q}_{n}$ as $k_d$-algebras, for $n\in \Z_{>0}$ and $d$ be a multiple of $n$. The path algebra $k_{d}\mathcal{Q}_{n}$ admits a $k_n$-linear automorphism $\G_{d,n }$ of order $d$, which is compatible with the action of $\Si_{d}$ on $k_{d}\otimes\Lambda_n$. Indeed, $\G_{d,n}: k_d  \mathcal{Q}_n\lra k_d  \mathcal{Q}_n$ is the authomorphism of $k_d$-algebras defined by $k\rho\mapsto\Si_d(k)\G_n(\rho)$, where $\G_n$ is the automorphism given by the ``reverse shift'' shown in Figure \ref{figuregamma}.

\begin{figure}
\begin{center}
 \begin{tikzpicture}
[->,>=stealth',shorten >=1pt,auto,node distance=2cm,thick,main node/.style=]
  \node[main node] (1) {$1$};   
  \node[main node] (6) [below of =1]{$1'$};   
  \node[main node] (2) [right of=1] {$2$};
  \node[main node] (7) [below of =2]{$2'$};   
  \node[main node] (3) [right of=2] {$\cdots$};
    \node[main node] (8) [below of =3]{$\cdots$};
  \node[main node] (4) [right of=3] {$n$};
  \node[main node] (5) [below of=4] {$n'$};
\path[every node/.style={font=\sffamily\small}]     
(1) edge node   {} (7)         
     edge node         {} (6)           
(2) edge node  {} (7)              
     edge node         {} (8)
(4) edge node  {} (5)              
   edge node         {} (6)
(3) edge node  {} (5) ; 
\draw[->, red,  dashed] (4) to[bend right]  (3);
\draw[->, red,  dashed] (3) to[bend right]  (2);
\draw[->, red,  dashed] (2) to[bend right]  (1);
\draw[->, red,  dashed] (1) to[bend left]  (4);
\draw[->, red,  dashed] (5) to[bend left]  (8);
\draw[->, red,  dashed] (8) to[bend left]  (7);
\draw[->, red,  dashed] (7) to[bend left]  (6);
\draw[->, red,  dashed] (6) to[bend right]  (5);
\end{tikzpicture}
 \end{center}
\caption{$\G_n$ is the ``reverse shift'' automorphism of order $n$ on the quiver $\mathcal{Q}_n$.} 
\label{figuregamma}
\end{figure}
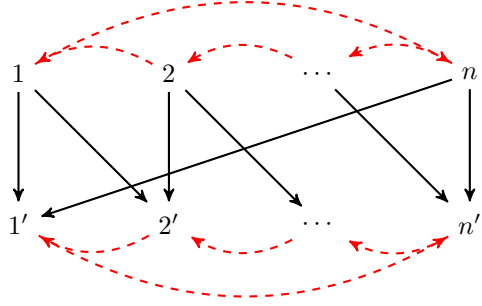


\section{Representations of twisted quivers and Galois descent}\label{main}

\subsection{Quivers with automorphisms}\label{isokd}
Our primary objective in this subsection is to establish the isomorphism of the $k_d$-algebras $k_d\otimes_k\Lambda_n$ and the path algebra $k_d\mathcal{Q}_n$. For this purpose, let $k_n=\C\parEn$ denote the degree $n$ extension of the field $k$ of Laurent series over the complex numbers. If we consider $k_n$ and $k_m$ as two extensions of $k$, the following proposition show that $k_n\otimes_kk_m$ is an extension of the field $k$ isomorphic to $\mcd(n,m)$ copies of $k_{\mcm(n,m)}$. 
\begin{prop}\label{obs1}
If $k_n=\C\parEn$, for any integer $n>0$, then  $k_n\otimes_k k_m\cong \underbrace{k_l\times\cdots\times k_l}_{g\text{-times}}$, where $l=\mcm(m,n)$ and $g=\mcd(m,n)$.
\end{prop}
\begin{proof}
See \cite[Section 2]{GR22}.
\end{proof}

In particular, if $d$ is a multiple of $n$, then we obtain the isomorphism $k_d\otimes_k k_n\cong \underbrace{k_d\times \cdots\times k_d}_{n-\text{times}}$ as $k_d$-algebras.

For every positive integer $n$, it is straightforward to verify that the set 
\begin{equation}\label{eq:basis}
  \left\{\B_j=\frac{1}{n\E}\sum_{l=0}^{n-1}\E^{\frac{n-l}{n}}\otimes\Si^{j-1}(\E^{\frac{l}{n}})\bigm\vert\,1\leq j\leq n\right\},  
\end{equation}

forms a basis of $k_n\otimes k_n$ as $k_n$-vector space. Additionally, the $\B_j$'s satisfy the equations $\B_j\B_l=\D_{jl}\B_j$, for all $1\leq j, l\leq n$. 
From of the definition $\Lambda_n=\left[\begin{array}{cc} k_n& k_n\oplus\,^{\Si_n}k_n\\0& k_n\end{array}\right]$, we obtain the isomorphism
$$
k_d\otimes_k\Lambda_n\cong \left[\begin{array}{cc} k_d\otimes k_n& k_d\otimes( k_n\oplus\,^{\Si_n}k_n)\\0& k_d\otimes k_n\end{array}\right]
$$
as $k_d$-algebras. Consequently, we can construct the sets of linearly independent matrices for this tensor product from the basis in (\ref{eq:basis}) in the following straightforward manner: 
\begin{eqnarray*}
\mathcal{B}_1&=& \left\{\B_{j1}=\frac{1}{n\E}\sum_{l=0}^{n-1}\E^\frac{n-l}{n}\otimes\Si^{j-1}(\E^\frac{l}{n})E_{11}\bigm\vert\,1\leq j\leq n\right\},\\
\mathcal{B}_2&=&\left\{\B_{j2}=\frac{1}{n\E}\sum_{l=0}^{n-1}\E^\frac{n-l}{n}\otimes\left[\begin{array}{cc} 0&(\Si^{j-1}(\E^\frac{l}{n}),0) \\0& 0\end{array}\right]\bigm\vert\,1\leq j\leq n\right\},\\
\mathcal{B}_3&=&\left\{\B'_{j1}=\frac{1}{n\E}\sum_{l=0}^{n-1}\E^\frac{n-l}{n}\otimes\left[\begin{array}{cc} 0&(0,\Si^{j-1}(\E^\frac{l}{n})) \\0& 0\end{array}\right]\bigm\vert\,1\leq j\leq n\right\},\\
\mathcal{B}_4&=& \left\{\B'_{j2}=\frac{1}{n\E}\sum_{l=0}^{n-1}\E^\frac{n-l}{n}\otimes\Si^{j-1}(\E^\frac{l}{n})E_{22}\bigm\vert\,1\leq j\leq n\right\},
 \end{eqnarray*}
where $E_{11}$ and $E_{22}$ are elements of the standard basis of $\M_{2\times 2}(k_d)$. Thus, the union of these four sets
\begin{eqnarray} \label{base1}
    \calb:=\calb_1\cup\calb_2\cup\calb_3\cup\calb_4,
\end{eqnarray}
forms naturally a basis of $k_d\otimes_k\Lambda_n$ as a $k_d$-algebra.

\begin{teo}\label{teoiso}
For a positive integer $n$ and a multiple $d$ of $n$, the $k_d$-algebras $k_d\otimes_k\Lambda_n$ and $k_d\mathcal{Q}_n$ are isomophic (as $k_d$-algebras), where $\Lambda_n=\left[\begin{array}{cc}
k_n&k_n\oplus\,^{\Si_n}k_n\\0&k_n\end{array}\right]$ and $\mathcal{Q}_n$ is the $n$-crown quiver (see Figure \ref{figurecrown}).
\end{teo}
\begin{proof}
Consider $\calb$, the basis of $k_d\otimes_k\Lambda_n$, in equation (\ref{base1}). It is straightforward to see that, for every $1\leq i, j\leq n$:
$$
\B_{j1}\B_{i2}=\D_{j,i}\,\B_{i2};\quad \B_{j1}\B'_{i1}=\D_{j-1,i}\,\B'_{i1};\quad \B_{j2}\B'_{i2}=\D_{j,i}\,\B_{j2};\quad \B'_{j1}\B'_{i2}=\D_{j,i}\,\B'_{j1}.
$$
Consequently, the morphism $k_d\otimes_k\Lambda\longrightarrow  k_d \mathcal{Q}_n$, defined by the assignment
$$
\begin{array}{cccc}\B_{j1}\mapsto j,& \B_{j2}\mapsto\A_{j},& \B'_{j1}\mapsto\A_{j'},&\B'_{j2}\mapsto j'\end{array}
$$

is an isomorphism of $k_d$-algebras, as it is easy to verify. Here, $j$ (resp. $j'$) is the $j$-th vertex (resp. $j'$-th vertex) and $\A_i:i\rightarrow i'$ and $\A_{i'}:(i-1)\rightarrow i'$ are arrows in the path algebra $k_d \mathcal{Q}_n$ of the $n$-crown $\mathcal{Q}_n$.
\end{proof}

Our second objective in this subsection is to show that Galois group $\langle\Si_d\rangle$ acts on $k_d\otimes \Lambda_n$. Indeed, this action is induced by the $k$-automorphism defined on $\calb$, $\Si_d\otimes\id: k_d\otimes_k\Lambda_n\rightarrow k_d\otimes_k\Lambda_n$, as follows:
\begin{eqnarray*}
  (\Si_d\otimes\id)(\B_{jk})=\B_{j-1,\,k},\mbox{ and } (\Si_d\otimes\id)(\B'_{jk})=\B'_{j-1,\,k}, \,1\leq k\leq 2.
\end{eqnarray*}
Thus, allow us define an action of the group $\langle\Si_d\rangle$ on $k_d \mathcal{Q}_n$ through the automorphism:
\begin{eqnarray*}
\G_{d,n}: k_d  \mathcal{Q}_n & \longrightarrow & k_d  \mathcal{Q}_n\\
k\rho & \longmapsto & \Si_d(k)\G_n(\rho),
\end{eqnarray*}
where $\G_n$ is the of order$-n$ automorphism of the $n$-crown $\mathcal{Q}_n$, see Figure \ref{figuregamma}. In this case the automorphism $\G_{d,n}$ involves $\G_n$, an automorphism of $\mathcal{Q}_n$, and  $\Si_d:k_d\lra k_d$, so we will said $\G_{d,n}$ is a ``reverse shift'' automorphism. 

The action of the group $\langle\Si_d\rangle$ on $k_d \mathcal{Q}_n$, is crucial for defining the skew group algebra $(k_d\mathcal{Q}_n)\langle\Si_d\rangle$ in Subsection \ref{GD}. 

\subsection{Invariant Representantions}\label{Invariantrep}
In this subsection, we investigate the behavior of the function resulting from applying $\id\otimes-$ to the morphism $\varphi_M\colon k_n^m\oplus \,^{\Si_n}k_n^{m}\lra k_n^{m}$. To begin, consider an arbitrary element $a\in k_n$. Following this, we can define the following $k_n$-morphisms $f_a:k_n\lra k_n$ and $g_a:k_n\lra\,^{\Si_n} k_n$ as follows $1\mapsto a$. The following lemma illustrates how the morphisms $\id\otimes f_a$ and $\id\otimes g_a$ act on the basis elements of $k_n\otimes k_n$.

\begin{lema}\label{lemma2.1}
Let $a\in k_n$, $f_a:k_n\lra k_n$ and $g_a:k_n\lra\,^{\Si_n} k_n$  be $k_n$-morphisms defined by $1\mapsto a$. We have $\id\otimes f_a:k_n\otimes k_n\lra k_n\otimes k_n$ and $\id\otimes g_a:k_n\otimes k_n\lra k_n\otimes\,^{\Si_n} k_n$, which are $k_n$-morphisms and act on the basis elements as follows:
 \begin{eqnarray*}
 ( \id\otimes f_a)(\B_j)&=&\Si_n^{n+1-j}(a)\B_j\\ 
  (\id\otimes g_a)(\B_j)&=&\Si_n^{n-j}(a)\B_{j+1}.
  \end{eqnarray*}
\end{lema}
\begin{proof}
Let $\B_j=\frac{1}{n\E}\sum_{l=0}^{n-1}\E^\frac{n-l}{n}\otimes \Si_n^{j-1}(\E^\frac{l}{n})$ be an element of the basis $\mathcal{B}$ in \ref{eq:basis}. It follows that:
\begin{eqnarray*}
   (\id\otimes f_a)(\B_j)&=&\frac{1}{n\E}\sum_{l=0}^{n-1}\id(\E^\frac{n-l}{n})\otimes f_a( \Si_n^{j-1}(\E^\frac{l}{n}))\\
   &=&\frac{1}{n\E}\sum_{l=0}^{n-1}\E^\frac{n-l}{n}\otimes a\Si_n^{j-1}(\E^\frac{l}{n}),
\end{eqnarray*}   
Since $a\in k_n$ can be expressed as a linear combination of elements on $k$, we say $a=\sum_{i=0}^{n-1}a_i\E^\frac{i}{n}$, and since $\Si_n(\E^\frac{l}{n})=\zeta_n^{l}\E^\frac{l}{n}$, we obtain that

\begin{eqnarray*}
   \frac{1}{n\E}\sum_{l=0}^{n-1}\E^\frac{n-l}{n}\otimes a\Si_n^{j-1}(\E^\frac{l}{n})&=&\frac{1}{n\E}\sum_{l=0}^{n-1}\E^\frac{n-l}{n}\otimes \left(\sum_{i=0}^{n-1}a_i\E^\frac{i}{n}\right) \zeta_n^{l(j-1)}\E^\frac{l}{n}\\
     &=&\frac{1}{n\E}\sum_{l=0}^{n-1}\E^\frac{n-l}{n}\otimes \left(\sum_{i=0}^{n-1}a_i\E^\frac{i+l}{n}\right) \zeta_n^{l(j-1)}.\\
\end{eqnarray*}
Now, due to the fact that $a_i\in k$ and summations properties, we have
\begin{equation}\label{eq1}
\begin{split}
   \frac{1}{n\E}\sum_{l=0}^{n-1}\E^\frac{n-l}{n}\otimes \left(\sum_{i=0}^{n-1}a_i\E^\frac{i+l}{n}\right) \zeta_n^{l(j-1)}&=\frac{1}{n\E}\sum_{l=0}^{n-1}\sum_{i=0}^{n-1}a_i\E^\frac{n-l}{n}\otimes \E^\frac{i+l}{n} \zeta_n^{l(j-1)}\\
   &=\frac{1}{n\E}\sum_{i=0}^{n-1}\sum_{l=0}^{n-1}a_i\E^\frac{n-l}{n}\otimes \E^\frac{i+l}{n} \zeta_n^{l(j-1)}\\
   &=\sum_{i=0}^{n-1}a_i\frac{1}{n\E}\left(\sum_{l=0}^{n-1}\E^\frac{n-l}{n}\otimes\E^\frac{i+l}{n}\zeta_n^{l(j-1)}\right).  
\end{split}
\end{equation}
If we fix an index $i$, we have $\sum_{l=0}^{n-1}\E^\frac{n-l}{n}\otimes\E^\frac{i+l}{n}\zeta_n^{l(j-1)}$ and using summations properties we obtain that
\begin{eqnarray*}
    \sum_{l=0}^{n-1}\E^\frac{n-l}{n}\otimes\E^\frac{i+l}{n}\zeta_n^{l(j-1)}&=&\sum_{l=0}^{n-(i+1)}\E^\frac{n-l}{n}\otimes\E^\frac{i+l}{n}\zeta_n^{l(j-1)}+\sum_{l=n-i}^{n-1}\E^\frac{n-l}{n}\otimes\E^\frac{i+l}{n}\zeta_n^{l(j-1)}\\
    &=&\sum_{l'=i}^{n-1}\E^\frac{n+i-l'}{n}\otimes\E^\frac{l'}{n}\zeta_n^{(l'-i)(j-1)}+\sum_{l'=0}^{i-1}\E^\frac{i-l'}{n}\otimes\E^\frac{n+l'}{n}\zeta_n^{(l'-i+n)(j-1)}.
\end{eqnarray*}
Nevertheless, $\E^\frac{i-l'}{n}\otimes\E^\frac{n+l'}{n}\zeta_n^{(l'-i+n)(j-1)}=\E^\frac{n+i-l'}{n}\otimes\E^\frac{l'}{n}\zeta_n^{(l'-i)(j-1)}=(\zeta_n^{-i(j-1)}\E^\frac{i}{n})\E^\frac{n-l'}{n}\otimes\E^\frac{l'}{n}\zeta_n^{(l')(j-1)},$ since $\E^\frac{n}{n}=\E,\, \zeta_n^{-i(j-1)}\in k$ and $\zeta_n^n=1$. Therefore 
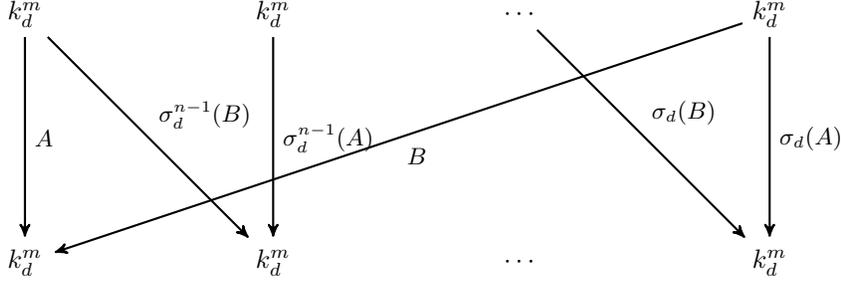
\begin{equation}\label{eq2}
\begin{split}
 \sum_{l=0}^{n-1}\E^\frac{n-l}{n}\otimes\E^\frac{i+l}{n}\zeta_n^{l(j-1)}&=(\zeta_n^{-i(j-1)}\E^\frac{i}{n})\sum_{l'=i}^{n-1}\E^\frac{n-l'}{n}\otimes\E^\frac{l'}{n}\zeta_n^{l'(j-1)}+\sum_{l'=0}^{i-1}\E^\frac{n-l'}{n}\otimes\E^\frac{l'}{n}\zeta_n^{(l')(j-1)}\\
&=(\zeta_n^{-i(j-1)}\E^\frac{i}{n})\sum_{l'=0}^{n-1}\E^\frac{n-l'}{n}\otimes\E^\frac{l'}{n}\zeta_n^{(l'-i)(j-1)}.
\end{split}
\end{equation}

We can substitute Equation (\ref{eq2}) into Equation (\ref{eq1}) and utilize the properties $\zeta_n^n=1$ and the definition of $\Si_n$ to obtain that   
\begin{eqnarray*}
\sum_{i=0}^{n-1}a_i\frac{1}{n\E}\left(\sum_{l=0}^{n-1}\E^\frac{n-l}{n}\otimes\E^\frac{i+l}{n}\zeta_n^{l(j-1)}\right)&=&\sum_{i=0}^{n-1}a_i\zeta_n^{(-i)(j-1)}\E^\frac{i}{n}\left(\frac{1}{n\E}\sum_{l'=0}^{n-1}\E^\frac{n-l'}{n}\otimes\E^\frac{l'}{n}\zeta_n^{l'(j-1)}\right)\\
&=&\sum_{i=0}^{n-1}a_i\zeta_n^{(i)(1-j)}\E^\frac{i}{n}\left(\B_j\right)\\
&=&\sum_{i=0}^{n-1}a_i\zeta_n^{(i)(n+1-j)}\E^\frac{i}{n}\left(\B_j\right)\\
&=&\sum_{i=0}^{n-1}\Si_n^{n+1-j}(a_i\E^\frac{i}{n})\left(\B_j\right)\\
&=&\Si_n^{n+1-j}(a)\left(\B_j\right).
\end{eqnarray*}
Hence, $(\id\otimes f_a)(\B_j)=\Si_n^{n+1-j}(a)\left(\B_j\right)$. A similar argument, utilizing the fact that $g_a(x)=\Si_n(x)a$, can be employed to compute $\id\otimes g_a$. This concludes the proof of the lemma.
\end{proof}

Let $a,\,b\in k_n$ be elements such that $(f_a,\,g_b):k_n\oplus\,^{\Si_n}k_n\lra k_n$ is a regular simple representation of $\Lambda_n$, and $d$ a multiple of $n$. Now, by application of $\id\otimes-$ to $(f_a,\,g_b)$, the induced morphism
$$
(\id\otimes f_a,\,\id\otimes g_b):(k_d\otimes k_n)\oplus (k_d\otimes\,^{\Si_n} k_n)\lra (k_d\otimes k_n),
$$
is, by Lemma \ref{lemma2.1}, such that
$$
\B_j\oplus 0\mapsto \Si_n^{n+1-j}(a)(\B_j)\quad\text{ and }\quad 0\oplus\B_j\mapsto \Si_n^{n-j}(a)\B_{j+1}.
$$
Since, by Proposition \ref{obs1}, $k_d\otimes k_n\cong \underbrace{k_d\times\cdots\times k_d}_{n\text{-times}}$, we have that $(\id\otimes f_a,\,\id\otimes g_b)$ can be identified with the following representation:
\begin{center}
\begin{tikzpicture}
[->,>=stealth',shorten >=1pt,auto,node distance=3cm,thick,main node/.style=]
  \node[main node] (1) {$k_d$};   
  \node[main node] (6) [below of =1]{$k_d$};   
  \node[main node] (2) [right of=1] {$k_d$};
  \node[main node] (7) [below of =2]{$k_d$};   
  \node[main node] (3) [right of=2] {$\cdots$};
    \node[main node] (8) [below of =3]{$\cdots$};
  \node[main node] (4) [right of=3] {$k_d$};
  \node[main node] (5) [below of=4] {$k_d$};
\path[every node/.style={font=\sffamily\small}]     
(1) edge node   {$a$} (6)         
     edge node         {$\Si_d^{n-1}(b)$} (7)           
(2) edge node  {$\Si_d^{n-1}(a)$} (7)              
(4) edge node  {$\Si_d(a)$} (5)              
   edge node         {$b$} (6)
   (3) edge node  {$\Si_d(b)$} (5) ;
\end{tikzpicture}
\end{center}
Similarly, let $\varphi_M\colon k_n^m\oplus \,^{\Si_n}k_n^{m}\lra k_n^{m}$ be a regular simple representation of $\Lambda_n$. Let $[AB]$ be the matrix associated with $\varphi_M$, where $A,\,B\in\M_{m\times m}(k_n)$, and $d:=\mcm(n,m)$. We have that $k_d\otimes_k \varphi_M$ can be identified with the following representation:

\begin{center}
\begin{tikzpicture}
[->,>=stealth',shorten >=1pt,auto,node distance=3.3cm,thick,main node/.style=]
  \node[main node] (1) {$k_d^m$};   
  \node[main node] (6) [below of =1]{$k_d^m$};   
  \node[main node] (2) [right of=1] {$k_d^m$};
  \node[main node] (7) [below of =2]{$k_d^m$};   
  \node[main node] (3) [right of=2] {$\cdots$};
    \node[main node] (8) [below of =3]{$\cdots$};
  \node[main node] (4) [right of=3] {$k_d^m$};
  \node[main node] (5) [below of=4] {$k_d^m$};
\path[every node/.style={font=\sffamily\small}]     
(1) edge node   {$A$} (6)         
     edge node         {$\Si_d^{n-1}(B)$} (7)           
(2) edge node  {$\Si_d^{n-1}(A)$} (7)              
(4) edge node  {$\Si_d(A)$} (5)              
   edge node         {$B$} (6)
   (3) edge node  {$\Si_d(B)$} (5) ;
\end{tikzpicture}
\end{center}
Here, $\Si_d$ is the automorphism defined by $\E^\frac{1}{d}\mapsto \zeta_d\E^\frac{1}{d}$, and $\Si_d^j(A)$ (resp. $\Si_d^j(B)$) denotes the matrix obtained by application of the morphism $\Si_d^j$ to each entry of the matrix $A$ (resp. $B$). This is a representation of our $n$-crown quiver $\mathcal{Q}_n$, see Figure \ref{figurecrown}.

From the ideas of Hubery's work in \cite[Section 3]{Hu02}, with appropriate modifications to our context, we define a module $\,^{\G_{d,n}}\mathcal{M}$ using the underlying vector space structure of $\mathcal{M}$, where $\mathcal{M}$ is a $ k_d \mathcal{Q}_n$-module. To this end, we introduce a new product $\ast$ as follows:
$$
p\ast m=\G_{d,n}^{-1}(p)m\text{, with}\, p\in k_d \mathcal{Q}_n.
$$
For each $f:\mathcal{M}\lra\mathcal{N}$ a module homomorphism, we obtain a homomorphism $\,^{\G_{d,n}}f:\,^{\G_{d,n}}\mathcal{M}\lra\,^{\G_{d,n}}\mathcal{N}$ as follows. Since $f$ is, in particular, a homomorphism of vector spaces, we have $\,^{\G_{d,n}}f=f$. Therefore
$$
f(p{\ast}m)=f(\G_{d,n}^{-1}(p)m)=\G_{d,n}^{-1}(p)f(m)=p{\ast}f(m).
$$
This defines an autoequivalence $F(\G_{d,n})$ on the category of (right) finitely generated modules on $ k_d \mathcal{Q}_n$ and satisfies $F(\G_{d,n}^r)=F(\G_{d,n})^r$, for any $r\in \mathbb{Z}$. Since $F(\G_{d,n})$ is an additive functor, we conclude that $\mathcal{M}$ is indecomposable if and only if $\,^{\G_{d,n}}\mathcal{M}$ is indecomposable.

The categories mod$-k_d \mathcal{Q}_n$ of right finitely generated $k_d$-modules and $\Rep (\mathcal{Q}_n, k_d)$ of $k_d$-linear representations of $\mathcal{Q}_n$ are known to be equivalent (see \cite[Sec. III]{ASS06}). Consequently, the functor $F(\G_{d,n})$ must also act on $\Rep(\mathcal{Q}_n, k_d)$.

Let $M=(V_i,\,V'_i,\,f_i,\,g_i)$ be a $ k_d$-representation of $\mathcal{Q}_n$, where $f_i:V_i\lra V'_i$ and $g_i:V_{i-1}\lra V'_i$. Let $\calm$ be the corresponding $k_d\mathcal{Q}_n$-module, so  $\calm$ has underlying $k_d$-vector space $V=\bigoplus_{j=1}^n(V_i\oplus V_i')$. Let $^{\G_{d,n}}\calm$ be the corresponding module and  $^{\G_{d,n}} M=(W_i,\,W'_i,\,\hat{f}_i,\,\hat{g}_i)$ its representation. We wish to describe $^{\G_{d,n}} M$ in terms of the original $M$. To this end, note that
$$
W_i=\epsilon_i\ast V=\G_{d,n}^{-1}(\epsilon_i)V=\epsilon_{\G^{-1}_n(i)}V=V_{\G^{-1}_n(i)},
$$
similarly $W'_i=V'_{\G^{-1}_n(i)}$. On one hand, the product on $W_i$ is defined by $a\ast w_j=\G_{d,n}^{-1}(a)w_i=\Si_{d}^{-1}(a)w_i$, for all $a\in k_d$ and $w_i\in W_i$. On the other hand, for each $f_i: V_i\lra V'_i$ (resp. $g_i:V_{i-1}\lra V_i'$) we can associate a matrix $A_i\in \M_{\dim V_i'\times \dim V_i}(k_d)$ (resp. $B_i\in \M_{\dim V_i'\times \dim V_{i-1}}(k_d)$). Thus, $\hat{f}_i:W_i\lra W'_i$ is defining by
$$
\hat{f}_i(w)=\A_i\ast w=\G_{d,n}^{-1}(\A_i)w=\A_{\G_n^{-1}(i)}w=A_{\G_n^{-1}(i)}w=\Si_d^{-1}(\Si_d(A_{\G_n^{-1}(i)}))w=\Si_d(A_{\G_n^{-1}(i)})\ast w, 
$$
where $\Si_d(A_i)$ means that $\Si_d$ is applied to each entry of the matrix $A_i$ and $\A_i$ is the arrow between $i$ and $i'$. Similarly, $\hat{g}_i(w)=\Si_d(B_{\G_n^{-1}(i)})\ast w$. Hence, if $M$ is a $ k_d$-representation of $\mathcal{Q}_n$ as below 
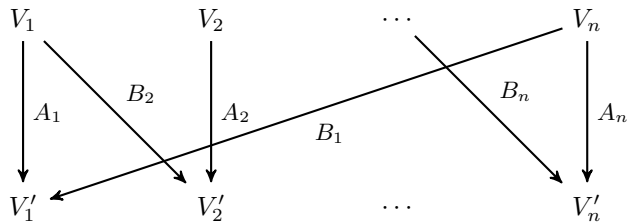
\begin{center}
\begin{tikzpicture}
[->,>=stealth',shorten >=1pt,auto,node distance=2.5cm,thick,main node/.style=]
  \node[main node] (1) {$V_1$};   
  \node[main node] (6) [below of =1]{$V'_1$};   
  \node[main node] (2) [right of=1] {$V_2$};
  \node[main node] (7) [below of =2]{$V'_2$};   
  \node[main node] (3) [right of=2] {$\cdots$};
    \node[main node] (8) [below of =3]{$\cdots$};
  \node[main node] (4) [right of=3] {$V_n$};
  \node[main node] (5) [below of=4] {$V'_n$};
\path[every node/.style={font=\sffamily\small}]     
(1) edge node   {$A_1$} (6)         
     edge node         {$B_2$} (7)           
(2) edge node  {$A_2$} (7)              
(4) edge node  {$A_n$} (5)              
   edge node         {$B_1$} (6)
(3) edge node  {$B_n$} (5) ;
\end{tikzpicture}
\end{center}
then the corresponding representation $\,^{\G_{d,n}} M$ is:
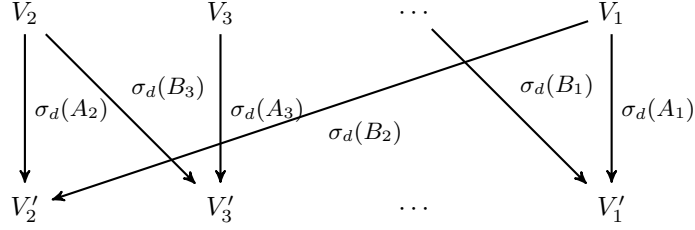
\begin{center}
\begin{tikzpicture}
[->,>=stealth',shorten >=1pt,auto,node distance=2.6cm,thick,main node/.style=]
  \node[main node] (1) {$V_2$};   
  \node[main node] (6) [below of =1]{$V'_2$};   
  \node[main node] (2) [right of=1] {$V_3$};
  \node[main node] (7) [below of =2]{$V'_3$};   
  \node[main node] (3) [right of=2] {$\cdots$};
    \node[main node] (8) [below of =3]{$\cdots$};
  \node[main node] (4) [right of=3] {$V_1$};
  \node[main node] (5) [below of=4] {$V'_1$};
\path[every node/.style={font=\sffamily\small}]     
(1) edge node   {$\Si_d(A_2)$} (6)         
     edge node         {$\Si_d(B_3)$} (7)           
(2) edge node  {$\Si_d(A_3)$} (7)              
(4) edge node  {$\Si_d(A_1)$} (5)              
   edge node         {$\Si_d(B_2)$} (6)
(3) edge node  {$\Si_d(B_1)$} (5) ;
\end{tikzpicture}
\end{center}
Similar to the approach in \cite{Hu02}, we will refer to $M$ as an \textit{isomorphically invariant representation}, denoted by $ii$-representation, if $\,^{\G_{d,n}}M\cong M$. We will say that $M$ is \textit{isomorphically invariant indecomposable} if it is not isomorphic to the proper direct sum  of two such isomorphically invariant representations.The general form of these isomorphically invariant indecomposables is demonstrated by the following lemma:

\begin{lema}\label{lemmaii} 
Let $M$ be a representation of $\mathcal{Q}_n$. The representation $M$ is an isomorphically invariant indecomposable if and only if there exists $N$ be an indecomposable representation such that
$$
M\cong N\oplus\,^{\G_{d,n}}N\oplus\,^{\G_{d,n}^2}N\oplus\cdots\oplus\,^{\G_{d,n}^{m-1}}N,
$$
and $m\geq 1$ is the smallest integer such that $\,^{\G_{d,n}^{m}}N\cong N$. 
\end{lema}
\begin{proof}
Suppose that $M$ is an $ii$-representation of $\mathcal{Q}_n$. By Krull-Remak-Schmidt theorem for $\mathcal{Q}_n$-representations, we can write $M$ as a direct sum of indecomposable representations as follows
$$
M\cong N_1\oplus N_2\oplus \cdots\oplus N_{m-1}.
$$
Given that $^{\G_{d,n}}M\cong M$ and $F({\G_{d,n}})$ is an additive functor, it must act (up to isomorphism) as a permutation of these indecomposable summands. Without loss of generality, we can assume that $N_i=\,^{\G^{i-1}_{d,n}}N_1$. If $N_j\cong N_1$, for some $1\leq j\leq m-1$, with $j$ minimum with this property. Then $\bigoplus_{i=1}^{j-1}N_i=\bigoplus_{i=1}^{j-1}\,^{\G^{i-1}_{d,n}}N_1$ is an $ii$-representation, since $F({\G_{d,n}})$ is additive.
    
Since $M$ and $\bigoplus_{i=1}^{j-1}N_i$ are $ii$-representations, it follows that $\bigoplus_{i=j}^{m}N_i$ is also an $ii$-representation, 
which contradicts the fact that $M$ is $ii$-indecomposable. Therefore $N_i\cong N_j$ if and only if $i=j$. Besides, $^{\G_{d,n}}N_{i}\cong N_{i+1}$ for $1\leq i\leq m-2$ and $^{\G_{d,n}}N_{m-1}\cong N_1$, \textit{i.e.}, $m$ is the smallest integer such that $\,^{\G_{d,n}^{m}}N_1\cong N_1$.

Conversely, let $N$ be an indecomposable representation, where $m\geq 1$ is the smallest integer such that $\,^{\G_{d,n}^{m}}N\cong N$. Since $F({\G_{d,n}})$ is an additive functor, we obtain that $\bigoplus_{i=1}^{m-1}\,^{\G^{i-1}_{d,n}}N$ is an $ii$-representation. The claim that $\bigoplus_{i=1}^{m-1}\,^{\G^{i-1}_{d,n}}N$ is $ii$-indecomposable follows from the fact that $m\geq 1$ is the smallest integer such that $\,^{\G_{d,n}^{m}}N\cong N$.
\end{proof}

\subsection{Skew group algebras} \label{GD}
Consider a field $\F$, an $\F$-algebra $\Lambda$ and a finite group $G$ acting on $\Lambda$. The \emph{skew group algebra} $\Lambda G$ can be defined (see \cite{RR85}). It has an $\F$-basis consisting of elements of the form $\lambda g$, where $\lambda\in \Lambda$ and $g\in G$. The product in $\Lambda G$  is defined as follows:
$$
\lambda g\cdot\mu g':=\lambda g(\mu)g\circ g'.
$$
An illustrative example of a skew group algebra is $\F_nG$, where $\F_n$ is a Galois extension of $ \F$ and $G=\langle\Si_n\rangle$ is its Galois group. A significant characterization of this algebra is provided by the following proposition. 
\begin{prop}\label{prop2} 
If $\F_n\mid \F$ is a Galois extension with Galois group $G=\langle\Si_n\rangle$, then the skew group algebra $\F_nG$ is isomorphic to $\M_{n\times n}(\F)$.
\end{prop}
\begin{proof}
See \cite[Appendix]{AG60} or \cite[Appendix A64]{Mi17}.
\end{proof}
In a broader context, when a finite cyclic group $G$ acts on $\Lambda$ the category of $\Lambda G$-modules can be viewed as a specific type of $\Lambda-$modules, as demonstrated by the following proposition.

\begin{prop}\label{obs3}
Let $G=\langle\Si\rangle$ be a finite cyclic group acting on an $\F$-algebra $\Lambda$, where $\F$ is a field. If $\Lambda G$ denotes the corresponding skew group algebra, then the category of $\Lambda G$-modules is equivalent to the category of pairs $(M, f)$, where $M$ is a $\Lambda$-module and $f:\,^{\Si}M\lra M$ is an isomorphism satisfying the following condition:
$$
f\circ\,^{\Si}f\circ\cdots\circ\,^{\Si^{n-1}}f=\id_M.
$$
\end{prop}
\begin{proof}
A sketch of the proof will be provided here. For further details, refer to \cite{RR85}. Consider an $\F$-algebra $\Lambda$ and a group $G=\langle\Si\rangle$ of order $n$. Recall that the skew group algebra $\Lambda G$ possesses an $\F$-basis consisting of elements of the form $\lambda\Si^r$, and the product is defined as follows: $\lambda\Si^r\mu\Si^s:=\lambda\Si^r(\mu)\Si^{r+s}$.

If $M$ is a $\Lambda G$-module, it can also be considered as a $\Lambda$-module under the product $\lambda m=\lambda \Si^0 m$. Additionally, the action of $\Si$ on $M$ is given by $\Si(m+n)=\Si m+\Si n$, and $\Si(\lambda m)=(\Si\lambda)m=(\Si(\lambda)\Si)m= \Si(\lambda)(\Si m).$ Now, define an isomorphism $f:\,^{\Si}M\lra M$ by mapping $m\mapsto \Si m$. This isomorphism satisfies the condition $f\circ\,^{\Si}f\circ\cdots\,\circ^{\Si^{n-1}}f=\id_M.$ Conversely, given a pair $(M,\,f)$, where $M$ is a $\Lambda$-module and $f:\,^{\Si}M\lra M$ is an isomorphism satisfying: $f\circ\,^{\Si}f\circ\cdots\,\circ^{\Si^{n-1}}f=\id_M$, we can define a $\Lambda G$-module structure on $M$ as follows:
$$
\lambda\Si^r m:=f\circ\,^{\Si}f\circ\cdots\,\circ^{\Si^{r-1}}f(\lambda m).
$$
This establishes the equivalence between the category of $\Lambda G-$modules and the category of pairs $(M,f)$.
\end{proof}

\begin{remark}\label{rem:n-crown}
{\em Consider the $n$-crown quiver, denoted by $\mathcal{Q}_n$, as described in Figure \ref{figurecrown}. As discussed in Subsection \ref{isokd} the group $\langle\Si_d\rangle$ acts on the path algebra $k_d\mathcal{Q}_n$ via the automorphism $\G_{d,n}$. This allows us to construct the skew group algebra $(k_d\mathcal{Q}_n)\langle\Si_d\rangle$, which has a basis consisting of elements of the form $\lambda\G_{d,n}^r$, where $\lambda$ represents a path in $k_d\mathcal{Q}_n$. The multiplication in this algebra is defined as follows:
$$\lambda\G_{d,n}^r\cdot \mu\G_{d,n}^s:=\lambda\G_{d,n}^r(\mu)\G_{d,n}^{r+s}.$$


Recalling that $k_d\mathcal{Q}_n\cong k_d\otimes \Lambda_n$ (Theorem \ref{teoiso}) and since $\langle\Si_d\rangle$ acts under this isomorphism, we have $(k_d\mathcal{Q}_n)\langle\Si_d\rangle\cong (k_d\otimes \Lambda_n)\langle\Si_d\rangle$. By Proposition \ref{prop2} it follows that 
$(k_d\mathcal{Q}_n)\langle \Si_d\rangle\cong\M_{d\times d}(\Lambda_n)$, as $k_n$-algebras. 

In particular, the categories of finitely generated modules $(k_d\mathcal{Q}_n)\langle \Si_d\rangle$-$\modu$ and $\Lambda_n$-$\modu$ are equivalents. 

On the other hand, consider $k_d\otimes X$ as a $k_d\otimes \Lambda_n$-module, for every $X\in\Lambda_n$-$\modu$. The action of $\langle\Si_d\rangle$ on $k_d\otimes X$ is given by $\,^{\Si_d}(k_d\otimes X):=\,^{\Si_d}k_d\otimes X$. Moreover, $\Si^{-1}_d\otimes\id_{X}:k_d\otimes X\lra \,^{\Si_d}(k_d\otimes X)$ is an isomorphism of $k_d\otimes \Lambda_n$-modules. This implies that $k_d\otimes X$ is an isomorphically invariant module. Furthermore, the isomorphism satisfies the following conditions:
$$
(\Si^{-1}_d\otimes\id_{X})\circ\,^{\Si_d}(\Si^{-1}_d\otimes\id_{X})\circ\cdots\circ\,^{\Si_d^{n-1}}(\Si^{-1}_d\otimes\id_{X})=\id_{k_d\otimes X}.
$$

If $X\in \Lambda_n$-mod is a regular simple module, then $\End_{\Lambda_n}(X)\cong k_m$ for some $m\in\Z_{>0}$. Consequently, if $d=\mcm(n, m)$ then, by Proposition \ref{obs1}, we have that
\begin{eqnarray}
\End_{k_d\otimes\Lambda_n}(k_d\otimes X)\cong k_d\otimes\End_{\Lambda_n}(X)\cong\underbrace{k_d\times\cdots\times k_d}_{m}.
\end{eqnarray}}
\end{remark}

\begin{lema}\label{obs6}
Consider a regular simple $\Lambda_n$-mod $X$ such that $\End_{\Lambda_n}(X)\cong k_m$ and $d=\mcm(n,m)$. Then,
$k_d\otimes X$ is isomorphic to a direct sum of regular simple $k_d\mathcal{Q}_n$-modules $N_1, N_2, \ldots, N_{m-1}$, where each $N_i$ satisfies $\Hom_{k_d\mathcal{Q}_n}(N_i,\,N_j)\cong\delta_{i,j}k_d$, with $\delta_{i,j}$ representing the Kronocker delta. Furthermore, for $1\leq i\leq m-2$, the action induces the isomorphisms $^{\G_{d,n}}N_i\cong N_{i+1}$, while $^{\G_{d,n}}N_{m-1}\cong N_1$.
\end{lema}
\begin{proof}
Let $X$ be a regular simple $\Lambda_n$-module, with $\End_{\Lambda_n}(X)\cong k_m$. Then 
$\End_{k_d\otimes\Lambda_n}(k_d\otimes X)\cong\underbrace{k_d\times\cdots\times k_d}_{m}.$ Additionally, $\tau_{k_d\mathcal{Q}_n}(k_d\otimes X)\cong k_d\otimes \tau_{\Lambda_n}(X)\cong k_d\otimes X$. Hence, $k_d\otimes X\cong N_1\oplus N_2\oplus\cdots\oplus N_{m-1},$ for regular simple $k_d\mathcal{Q}_n$-modules, where each $N_i$ is a regular simple $k_d\mathcal{Q}_n$-module, with $\Hom_{k_d\mathcal{Q}_n}(N_i,\,N_j)\cong\delta_{i,j}k_d$, for $\delta_{i,j}$ the Kronocker delta.
    
From the description in Remark \ref{rem:n-crown}, we establish that $k_d\otimes X\cong \,^{\Si_d}(k_d\otimes X)$. This allows us to deduce the isomorphism, $\bigoplus_{i=1}^{m-1}N_i\cong \,^{\G_{d,n}}(\bigoplus_{i=1}^{m-1}N_i)$. Since the functor $F({\G_{d,n}})$ is additive, we can assume that $^{\G_{d,n}}N_i\cong N_{i+1}$, for $1\leq i\leq m-2$, and $^{\G_{d,n}}N_{m-1}\cong N_1$, thus completing the proof.
\end{proof}

\begin{teo}\label{lemma3.7}
  Consider an integer $n>0$. There exists a bijective correspondence between the isomorphism classes of regular simple $\Lambda_n$-modules with endomorphism ring isomorphic to $k_m$, where $d=\mcm(n,m)$, and the orbits under $\langle\Si_d\rangle$ of regular simple $k_d\mathcal{Q}_n$-modules $N$ satisfying the following conditions:
\begin{enumerate}
    \item $\End_{k_d\mathcal{Q}_n}(N)\cong k_d$,
\item $\,^{\Si_d^m}N\cong N\mbox{ y }\,^{\Si_d^j}N\not\cong N$  with $j=1,2,\cdots,m-1$.
\end{enumerate}
\end{teo}
\begin{proof}
Let $X\in \Lambda_n$-mod be regular simple. By Lemma \ref{obs6}, there exists a $k_d\mathcal{Q}_n$-module $N$ that satisfies both of the given conditions. Conversely, if $N$ is a module that satisfies both conditions, we can consider the module $\bigoplus_{j=0}^{m-1}\,^{\Si^j}N$ as a $(k_d\mathcal{Q}_n)\langle\Si_d\rangle$-module. Since $(k_d\mathcal{Q}_n)\langle\Si_d\rangle$-modules are equivalent to $\Lambda_n$-modules, there exists a regular simple $\Lambda_n-$module associated with $\bigoplus_{j=0}^{m-1}\,^{\Si_d^j}N$. 
\end{proof}

\subsection{Homogenous and non homogenous representations}\label{hnhrep}
As established in Theorem \ref{lemma3.7}, our primary focus is on identifying representatives of regular simple representations of the 
$n$-crown $\mathcal{Q}_n$. To achieve this, we begin by noting that $\G_{d,n}$ commutes with $\tau$, the Auslander-Reiten translate (see \cite[Lemma 4.1]{RR85}). Consequently, $\G_{d,n}$ acts on both the preprojective and preinjective components of the Auslander-Reiten quiver.

The indecomposable regular representations are precisely those with defect $0$, and $\tau$ operates with finite period on each tube of the Auslander-Reiten quiver. Each regular indecomposable representation exhibits regular serial behavior and is uniquely determined, up to isomorphism, by its regular top and regular length. Therefore, the actions of both $\G_{d,n}$ and $\tau$ are determined by their actions on the regular simples.

Now, following \ref{lemmabands}, we consider the regular simple representations $M_a$, for $n\in\Z_{>0}$ and $d$ be a multiple of $n$:
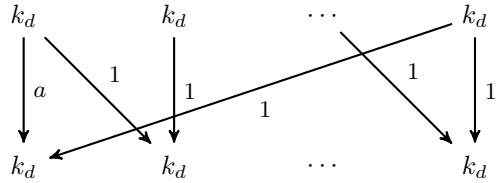
\begin{center}
\begin{tikzpicture}
[->,>=stealth',shorten >=1pt,auto,node distance=2cm,thick,main node/.style=]
  \node[main node] (1) {$k_d$};   
  \node[main node] (6) [below of =1]{$k_d$};   
  \node[main node] (2) [right of=1] {$k_d$};
  \node[main node] (7) [below of =2]{$k_d$};   
  \node[main node] (3) [right of=2] {$\cdots$};
    \node[main node] (8) [below of =3]{$\cdots$};
  \node[main node] (4) [right of=3] {$k_d$};
  \node[main node] (5) [below of=4] {$k_d$};
\path[every node/.style={font=\sffamily\small}]     
(1) edge node   {$a$} (6)         
     edge node         {$1$} (7)           
(2) edge node  {$1$} (7)              
(4) edge node  {$1$} (5)              
   edge node         {$1$} (6)
   (3) edge node  {$1$} (5) ;
\end{tikzpicture}
\end{center}

with $a\neq 0$. Considering $M_a$ and $M_b$ as representations as described above, they are isomorphic if and only $a=b$.  As we observed in Subsection \ref{Invariantrep}, $\G_{d,n}$ acts on $M_a$ as follows: 
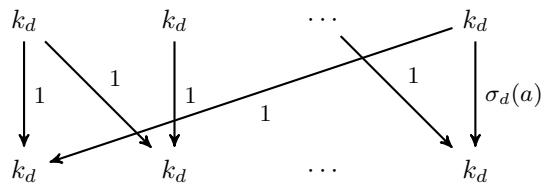
\begin{center}
\begin{tikzpicture}
[->,>=stealth',shorten >=1pt,auto,node distance=2cm,thick,main node/.style=]
  \node[main node] (1) {$k_d$};   
  \node[main node] (6) [below of =1]{$k_d$};   
  \node[main node] (2) [right of=1] {$k_d$};
  \node[main node] (7) [below of =2]{$k_d$};   
  \node[main node] (3) [right of=2] {$\cdots$};
    \node[main node] (8) [below of =3]{$\cdots$};
  \node[main node] (4) [right of=3] {$k_d$};
  \node[main node] (5) [below of=4] {$k_d$};
\path[every node/.style={font=\sffamily\small}]     
(1) edge node   {$1$} (6)         
     edge node         {$1$} (7)           
(2) edge node  {$1$} (7)              
(4) edge node  {$\Si_d(a)$} (5)              
   edge node         {$1$} (6)
   (3) edge node  {$1$} (5) ;
\end{tikzpicture}
\end{center}

We denote this representation by $^{\G_{d,n}}M_a$. Since $a\neq 0$, it follows that $\Si_d(a)\neq 0$ and therefore the morphism $f_{(\Si_d(a))^{-1}}:k_d\lra k_d$ is an isomorphism. Consequently, we have that $^{\G_{d,n}}M_a\cong M_{\Si_d(a)}$, and further,
$^{\G_{d,n}^j}M_a\cong M_{\Si_d^j(a)}$, for all $1\leq j\leq n-1$. Thus, we are interested in elements $a\in k_n$ that are \textit{generic}, meaning that $\Si_d^j(a)\neq a$, for $1\leq j\leq n-1$. Thus, for $a\in k_m$,  a generic element, we have the representation:
$$
M:=\bigoplus_{j=0}^{m-1}\,^{\G_{d,n}^j}M_a
$$
which is an $ii-$indecomposable representation. These representations are of particular interest to us. As will be shown later, they provide the pairs described in \ref{obs3}. As demonstrated in Propositions \ref{propmaless} and \ref{propmagreater}, a crucial aspect in determining these pairs is the fact that the $\frac{n-m}{\mcd(n,m)}$-th power of these generic elements has a $\frac{d}{m}$-root in $k_d$. To establish this, we prove the following Lemma, which is slightly more general:

\begin{lema}\label{nth-root}
Consider the field of complex Laurent series $\C\lp x\rp$ and let $n$ be a positive integer. For any element $a\in\C\lp x\rp$, there exists an element $b\in\C\lp x\rp$ such that $a=b^nx^m$, for some $0\leq m\leq n-1$ and $\cogrado(a)\equiv m\,(\modu n)$.
\end{lema}
\begin{proof}
Consider the element $b\in\C\lp x\rp$, let say $b=\sum_{j=0}^{\infty}b_jx^j$. Suppose that $b^n=\sum_{j=0}^{\infty}c_jx^j$, for some $n>0$. We will demonstrate that this leads to a recurrence relation for the coefficients $c_j$ in terms of the coefficients $b_i$. To achieve this, a common practice is to view $b$ and $b^n$ as ``functions'' and then obtain the derivative of $b^n$ with respect to $x$. This will allow us to obtain the recurrence relations. More precisely, by definition of derivates, $(b^n)'=nb^{n-1}b'$, then $(b^n)'b=nb^{n}b'$. Since $(b^n)'=\sum_{j=1}^{\infty}jc_jx^{j-1}=\sum_{j=0}^{\infty}(j+1)c_{j+1}x^j$ and $b'=\sum_{j=0}^{\infty}(j+1)b_{j+1}x^j$, we obtain that
$$
\sum_{j=0}^{\infty}(j+1)c_{j+1}x^j\sum_{j=0}^{\infty}b_jx^j=(b^n)'b=nb^{n}b'=n\sum_{j=0}^{\infty}c_jx^j\sum_{j=0}^{\infty}(j+1)b_{j+1}x^j,
$$
equating coefficients, we obtain a recurrence relation for $c_j$. 
    
In a similar way, if $a=\sum_{j=0}^{\infty}a_jx^j$, with $a_0\neq0$, we can define $b_j$ recursively until we obtain $b^n=a$, for $b=\sum_{j=0}^{\infty}b_jx^j\in \C\lp x\rp$. Furthermore, if $a\in\C\lp x\rp$ is such that $\cogrado(a)$ is a multiple of $n$, then $a=cx^{\cogrado(a)}$, for some $c\in\C\lp x\rp$, and $a=cx^{\cogrado(a)}=b^nx^{\cogrado(a)}=(bx^{\cogrado(a)/n})^n.$

Hence, for $a\in\C\lp x\rp$, with $\cogrado(a)$ is a multiple of $n$, there exists $b\in\C\lp x\rp$ such that $a=b^n$. More generally, if $a\in\C\lp x\rp$ is such that $\cogrado(a)\equiv m\,(\modu n)$, with $1\leq m\leq n-1$, then there exists $c\in\C\lp x\rp$, with $\cogrado(c)$ a multiple of $n$ such that $a=cx^m=b^nx^m.$
\end{proof}

\begin{remark}
{\em An interesting and useful remark is the following
\begin{eqnarray}\label{Eq7}
    \Pi_{i=0}^{n-1}\Si_n^i(\E^{\frac{j}{n}})=\Pi_{i=0}^{n-1}\zeta_n^{ij}(\E^{\frac{j}{n}})=(\zeta_n^{\left(\sum_{i=0}^{n-1}i\right)j})\E^j=\E^j.
\end{eqnarray}}
\end{remark}

Now, our main objective is to prove that for $a\in k_m$, a generic element, there exists an isomorphism $f:\bigoplus_{j=0}^{m-1}\,^{\G_{d,n}^j}M_a\lra\,^{\G_{d,n}}\left(\bigoplus_{j=0}^{m-1}\,^{\G_{d,n}^j}M_a\right)$, such that the induced automorphism $^{\G_{d,n}^{l-1}}f\circ\cdots\circ^{\G_{d,n}}f\circ f$ of $M$ is the identity for some positive integer $l$. To this end, we will consider the following three cases: $m=n$, $m<n$ and $m>n$.
\begin{prop}\label{propma}
Let $M=\bigoplus_{j=0}^{n-1}\,^{\G_{d,n}^j}M_a$ be an $ii$-indecomposable representation. Then there exists an isomorphism $f:M\lra\,^{\G_{d,n}}M$ such that the automorphism $^{\G_{d,n}^{n-1}}f\circ\cdots\circ^{\G_{d,n}}f\circ f$ of $M$ is the identity.
\end{prop}
\begin{proof}
Consider the $ii-$indecomposable representation $M$ as the direct sum $\bigoplus_{j=0}^{n-1}\,^{\G_{d,n}^j}M_a$, where $M_a$ is a regular simple representation defined earlier for $a\in k_n$, a generic element. We can assume that $M$ has the form
\begin{center}
\begin{tikzpicture}
[->,>=stealth',shorten >=1pt,auto,node distance=3cm,thick,main node/.style=]
  \node[main node] (1) {$k_d^n$};   
  \node[main node] (6) [below of =1]{$k_d^n$};   
  \node[main node] (2) [right of=1] {$k_d^n$};
  \node[main node] (7) [below of =2]{$k_d^n$};   
  \node[main node] (3) [right of=2] {$\cdots$};
    \node[main node] (8) [below of =3]{$\cdots$};
  \node[main node] (4) [right of=3] {$k_d^n$};
  \node[main node] (5) [below of=4] {$k_d^n$};
\path[every node/.style={font=\sffamily\small}]     
(1) edge node   {$A_1$} (6)         
     edge node         {$\id_n$} (7)           
(2) edge node  {$A_n$} (7)              
(4) edge node  {$A_2$} (5)              
   edge node         {$\id_n$} (6)
   (3) edge node  {$\id_n$} (5) ;
\end{tikzpicture}
\end{center}
with $A_j=\diag(1,\cdots,\Si_n^{j-1}(a)\,\cdots,1)$, that is, the diagonal matrix with 1 in all entries except the $\G_n^{j-1}(1)$-th entry. Thus, the induced representation $\,^{\G_{d,n}}M$ is given by:
\begin{center}
\begin{tikzpicture}
[->,>=stealth',shorten >=1pt,auto,node distance=3.3cm,thick,main node/.style=]
  \node[main node] (1) {$k_d^n$};   
  \node[main node] (6) [below of =1]{$k_d^n$};   
  \node[main node] (2) [right of=1] {$k_d^n$};
  \node[main node] (7) [below of =2]{$k_d^n$};   
  \node[main node] (3) [right of=2] {$\cdots$};
    \node[main node] (8) [below of =3]{$\cdots$};
  \node[main node] (4) [right of=3] {$k_d^n$};
  \node[main node] (5) [below of=4] {$k_d^n$};
\path[every node/.style={font=\sffamily\small}]     
(1) edge node   {$\Si_n(A_n)$} (6)         
     edge node         {$\id_n$} (7)           
(2) edge node  {$\Si_n(A_{n-1})$} (7)              
(4) edge node  {$\Si_n(A_1)$} (5)              
   edge node         {$\id_n$} (6)
   (3) edge node  {$\id_n$} (5) ;
\end{tikzpicture}
\end{center}

Now, we consider the linear morphism $k_d^n\longrightarrow k_d^n$ such that its associated matrix is given by
\begin{eqnarray}\label{eqen}
    (E_n)_{ij}=\left\{\begin{array}{ccc}\delta_{i,\,j-1}&\mbox{if}&1<j\leq n,\\\delta_{i,\,1}&\mbox{if}&j=n,\end{array}\right.
\end{eqnarray}
where $\delta_{i,j}$ is the Kronocker delta, i.e., 
$$
E_n=\left[\begin{array}{ccccc}0&0&\cdots&0&1\\ 1&0&\cdots&0&0\\0&1&\cdots&0&0\\\vdots&\vdots&\ddots&\vdots&\vdots\\0&0&\cdots&1&0\end{array}\right]
$$
Thus, we can define the morphism $f=((f_i,f_{i'}))_{1\leq i\leq n}$ from $M$ to $\,^{\G_{d,n}}M$, where
$f_i=f_{i'}=E_n$ as above. It is easy to see that
\begin{center}
\begin{tikzpicture}
[->,>=stealth',shorten >=1pt,auto,node distance=2cm,thick,main node/.style=]
  \node[main node] (1) {$k_d^n$};   
  \node[main node] (2) [right of=1] {$k_d^n$};
  \node[main node] (3) [below of=1] {$k_d^n$};
  \node[main node] (4) [below of=2] {$k_d^n$};
 \path[every node/.style={font=\sffamily\small}]     
(1) edge node   {$A_j$} (2)          
(1) edge node    {$E_n$} (3)           
(3) edge node  {$\Si_n(A_{j-1})$} (4)
(2) edge node  {$E_n$}  (4)              
;
\end{tikzpicture}
\end{center}
meanwhile, for $j=1$,
\begin{center}
\begin{tikzpicture}
[->,>=stealth',shorten >=1pt,auto,node distance=2cm,thick,main node/.style=]
  \node[main node] (1) {$k_d^n$};   
  \node[main node] (2) [right of=1] {$k_d^n$};
  \node[main node] (3) [below of=1] {$k_d^n$};
  \node[main node] (4) [below of=2] {$k_d^n$};
 \path[every node/.style={font=\sffamily\small}]     
(1) edge node   {$A_1$} (2)          
(1) edge node    {$E_n$} (3)           
(3) edge node  {$\Si_n(A_{n})$} (4)
(2) edge node  {$E_n$}  (4)              
;
\end{tikzpicture}
\end{center}
Consequently, $\,^{\G_{d,n}}f=((g_i,g_{i'}))_{1\leq i\leq n}$, where $g_i=f_{i+1}, \,g_{i'}=f_{(i+1)'}\textit{ for }1\leq i\leq n-1\textit{ and }g_n=f_{1},\,g_{n'}=f_{1'}.$ In conclusion, for the morphism $f$, we have that $\,^{\G_{d,n}}f=f$. Thus 
$$^{\G_{d,n}^{n-1}}f\circ\cdots\circ^{\G_{d,n}}f\circ f=f\circ f\circ\cdots\circ f=((E_n^n,E_n^n))_{1\leq i\leq n}.$$ 

Nevertheless, $E_n^n=\id_n$, is the identity linear map. Hence, we can conclude that $f=((f_i,f_{i'}))_{1\leq i\leq n}$ is the morphism satisfying the desired properties. 
\end{proof}

\begin{prop}\label{propmaless}
Let $m$ be a positive integer, with $m<n$, and $M=\bigoplus_{j=0}^{m-1}\,^{\G_{d,n}^j}M_a$ be an $ii$-indecomposable representation. Then there exists an isomorphism $f:M\lra\,^{\G_{d,n}}M$ such that the induced automorphism $^{\G_{d,n}^{d-1}}f\circ\cdots\circ^{\G_{d,n}}f\circ f$ of $M$ is the identity.
\end{prop}
\begin{proof}
Consider the $ii-$indecomposable representation $M$ as the direct sum $\bigoplus_{j=0}^{m-1}\,^{\G_{d,n}^j}M_a$, where $M_a$ is a regular simple representation defined earlier for a generic element $a\in k_m$. We assume that $M$ has the form depicted in Figure \ref{figurema}. 
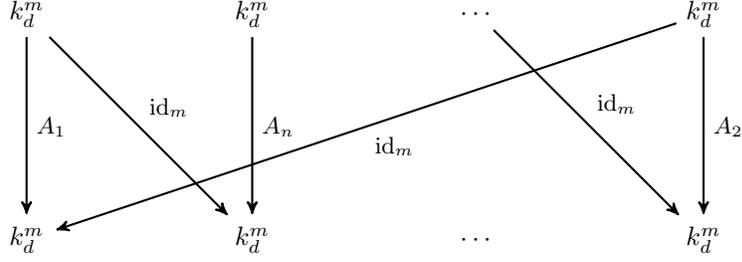
\begin{figure}
    \centering
\begin{tikzpicture}
[->,>=stealth',shorten >=1pt,auto,node distance=3cm,thick,main node/.style=]
  \node[main node] (1) {$k_d^m$};   
  \node[main node] (6) [below of =1]{$k_d^m$};   
  \node[main node] (2) [right of=1] {$k_d^m$};
  \node[main node] (7) [below of =2]{$k_d^m$};   
  \node[main node] (3) [right of=2] {$\cdots$};
    \node[main node] (8) [below of =3]{$\cdots$};
  \node[main node] (4) [right of=3] {$k_d^m$};
  \node[main node] (5) [below of=4] {$k_d^m$};
\path[every node/.style={font=\sffamily\small}]     
(1) edge node   {$A_1$} (6)         
     edge node         {$\id_m$} (7)           
(2) edge node  {$A_n$} (7)              
(4) edge node  {$A_2$} (5)              
   edge node         {$\id_m$} (6)
   (3) edge node  {$\id_m$} (5) ;
\end{tikzpicture}
    \caption{A representation of $\bigoplus_{j=0}^{m-1}\,^{\G_{d,n}^j}M_a$.}
    \label{figurema}
\end{figure}
In this case, $A_j=\diag(1,\cdots,\Si_d^{j-1}(a)\,\cdots,1)$ represents the diagonal matrix with 1 in all entries except the $\G_m^{j-1}(1)$-th entry, for $1\leq j\leq m$, and $A_j=\id_m$, for $m+1\leq j\leq n$. By Lemma \ref{nth-root}, there exists $b\in k_d$ such that $b^{d/m}=a^{\frac{n-m}{\mcd(n,m)}}$. We analyze the following cases: 
\begin{itemize}
\item If $\mcm(m,n)=n$, it suffices to take $f=((E_m,E_m))_{1\leq i\leq n}$. 
\item If $\mcm(m,n)\neq n$, we can consider $f_i,\,f_{i'}\in\M_{m\times m}(k_d)$, for $1\leq i\leq n$, the pair of matrizes defined as follows:
$$
\begin{array}{ccc}f_{i}=\diag(1,x_i,\cdots,1)E_m&\mbox{and}&f_{i'}=\left\{\begin{array}{ll}f_{i-1}&\mbox{if }2\leq i\leq n\\f_{n}&\mbox{if }i=1,\end{array}\right.\end{array}
$$
where $E_m$ is the matrix described in Equation (\ref{eqen}), $x_i=ab^{-1}$, if $1\leq i \leq n-m$, and $x_i=b^{-1}$, otherwise. 
\end{itemize}

Now, as established in the proof of Proposition \ref{propma}, we obtain that $\Si_d(A_{j-1})E_m=E_m A_j$, for $1\leq j<m$, and $\Si_d(A_{m})E_m=E_m A_1$. Therefore:
$$
\Si_d(A_{j-1})\diag(1,b^{-1},\cdots,1)E_m=\diag(1,b^{-1},\cdots,1)\Si_d(A_{j-1})E_m=\diag(1,b^{-1},\cdots,1)E_mA_j,
$$
$$
\Si_d(A_{m})\diag(1,b^{-1},\cdots,1)E_m=\diag(1,b^{-1},\cdots,1)\Si_d(A_{m})E_m=\diag(1,b^{-1},\cdots,1)E_mA_1,
$$
and, for $m+1\leq j\leq n-1$, it follows that
$$
\Si_d(A_j)f_{n+1-j}=\id_m\diag(1,ab^{-1},\cdots,1)E_m=\diag(1,ab^{-1},\cdots,1)E_m\id_m=f_{n-j}A_{j+1}.
$$
On the other hand, we have that
\begin{eqnarray*}
\Si(A_n)f_1&=&\id_m\diag(1,ab^{-1},\cdots,1)E_m\\
&=&\diag(1,b^{-1},\cdots,1)\diag(1,a,\cdots,1)E_m\\
&=&\diag(1,b^{-1},\cdots,1)E_m\diag(a,1,\cdots,1)\\
&=&\diag(1,b^{-1},\cdots,1)E_m A_1.
\end{eqnarray*}
Thus, we prove that the pair $f=((f_i,f_{i'}))_{1\leq i\leq n}$ defines a morphism from $M$ to $\,^{\G_{d,n}}M$. Consequently, 
$$
\,^{\G_{d,n}}f=((\Si_d(f_{\G^{-1}_n(i)}),\Si_d(f_{\G^{-1}_n(i')})))_{1\leq i\leq n},\quad\text{and}\quad 
\,^{\G_{d,n}^j}f=((\Si_d^j(f_{\G^{-j}_n(i)}),\Si_d^j(f_{\G^{-j}_n(i')})))_{1\leq i\leq n}.
$$

Now, we start grouping the first $m$-morphisms to obtain
\begin{eqnarray*}
\,^{\G_{d,n}^{m-1}}f\circ \cdots\circ\,^{\G_{d,n}}f\circ f&=&((\Si_d^{m-1}(f_{\G^{-j}_n(i)})\cdots\Si_d(f_{\G^{-1}_n(i)})f_i,\Si_d^{m-1}(f_{\G^{-j}_n(i')})\cdots\Si_d(f_{\G^{-1}_n(i')})f_{i'})))_{1\leq i\leq n}\\
&=&((\diag(x_i,\cdots,\,\Si_d^{m-1}(x_{\G_n^{1-m}(i)})),g_{i',1}))_{1\leq i\leq n},
\end{eqnarray*}
where $g_{i',1}:=g_{\G_n(i),1}$ and $g_{i,1}=\diag(x_i,\Si_d(x_{\G_n^{-1}(i)}),\cdots,\,\Si_d^{m-1}(x_{\G_n^{1-m}(i)}))$. Then, in the more general case, for $1\leq l\leq d/m$, it follows that
\begin{eqnarray*}
\,^{\G_{d,n}^{lm-1}}f\circ \cdots\circ\,^{\G_{d,n}^{m(l-1)}}f&=&((g_{i,l},g_{i',l}))_{1\leq i\leq n},
\end{eqnarray*}
where $g_{i',l}:=g_{\G_n(i),l}$ and $g_{i,l}=\diag(\Si_d^{m(l-1)}(x_{\G_n^{m-ml}(i)}),\cdots,\,\Si_d^{ml-1}(x_{\G_n^{1-ml}(i)}))$. 

From the foregoing computations, we derive that $g_{1,d/n}g_{1,(d/n)-1}\cdots g_{1,1}=\diag(y_1,\cdots,y_m),$ where
$$
y_j=\Si_d^{m(\frac{d}{m}-1)+j-1}(x_{\G_n^{m(1-\frac{d}{m})+1-j}(1)})\cdots\Si_d^{m+j-1}(x_{\G_n^{1-(m+j)}(1)})\Si_d^{j-1}(x_{\G_n^{1-j}(1)}).
$$
Since either $x_i=ab^{-1}$ or $x_i=b^{-1}$, a well-known combinatorial problem determines that the number of factors of $y_j$ for which some $x_{\G_n^{m(1-l)+1-j}(1)}$ is equal $ab^{-1}$, is given by  $\frac{n-m}{\mcd(n,m)}$. Now, using the fact of $a\in k_m$ and Equation \ref{Eq7} we conclude that
$$
y_j=\Si_d^{j-1}(a^{\frac{n-m}{\mcd(n,m)}})\Pi_{i=0}^{\frac{d}{m}-1}\Si_d^{mi+j-1}(b^{-1})=\Si_d^{j-1}(a^{\frac{n-m}{\mcd(n,m)}})\Si_d^{j-1}\left(\Pi_{i=0}^{\frac{d}{m}-1}\Si_d^{mi}(b^{-1})\right)=\Si_d^{j-1}(a^{\frac{n-m}{\mcd(n,m)}}b^{-\frac{d}{m}}).
$$
Consequently, $g_{1,d/n}g_{1,(d/n)-1}\cdots g_{1,1}=\id_m$. Moreover, it can be similarly proven that $g_{i,d/n}g_{i,(d/n)-1}\cdots g_{i,1}=\id_m$, for $1\leq i\leq n$. In conclusion, $^{\G_{d,n}^{d-1}}f\circ\cdots\circ^{\G_{d,n}}f\circ f$ is the identity, which completes the proof of the proposition.
\end{proof}

\begin{prop}\label{propmagreater}
Let $m$ be an positive integer, with $m>n$, and $M=\bigoplus_{j=0}^{m-1}\,^{\G_{d,n}^j}M_a$ be an $ii$-indecomposable representation. Then there exists an isomorphism $f:M\lra\,^{\G_{d,n}}M$ such that the induced automorphism $^{\G_{d,n}^{d-1}}f\circ\cdots\circ^{\G_{d,n}}f\circ f$ of $M$ is the identity.
\end{prop}
\begin{proof}
As in the proof of Proposition \ref{propmaless} we consider the $ii-$indecomposable representation $M$ as the direct sum $\bigoplus_{j=0}^{m-1}\,^{\G_{d,n}^j}M_a$, where $M_a$ is a regular simple representation and $a\in k_m$, a generic element. We can assume that $M$ has the form of Figure \ref{figurema}. In this case, $A_i=\diag(x_{i,1},\cdots,\,x_{i,m})\in\M_{m\times m}(k_d)$, where
$$
x_{i,j}=\left\{\begin{array}{cl}
     \Si_d^{m+1-j}(a)&\mbox{if }j\equiv\G_n^{i-1}(1)  \\
     1 & \mbox{otherwise.}
\end{array}\right.
$$
Let $l$ be a positive integer such that $l\equiv m (\text{mod}\, n)$ and we consider $b\in k_d$, such that $b^{d/m}=\Si_d(a)^{\frac{n-l}{\mcd(n,l)}}$. If $\mcm(m,n)=m$, it suffices to take $f=((E_m,E_m))_{1\leq i\leq n}$. Otherwise, that is $\mcm(m,n)\neq m$, we can consider  $f_i,\,f_{i'}\in\M_{m\times m}(k_d)$, for $1\leq i\leq n$, the pair of matrizes defined as follows:
$$
\begin{array}{ccc}f_{i}=\diag(x_i,1,\cdots,1)E_m&\mbox{and}&f_{i'}=\left\{\begin{array}{ll}f_{i-1}&\mbox{if }2\leq i\leq n\\f_{n}&\mbox{if }i=1\end{array}\right.\end{array}
$$
where $E_m$ is the matrix described in Equation (\ref{eqen}), $x_i=\Si_d(a)b^{-1}$, if $l\leq i \leq n-1$,  and $x_i=b^{-1}$, otherwise. Demonstrating that the pair $f=((f_i,f_{i'}))_{1\leq i\leq n}$ is the morphism with the desired properties follows a similar line of reasoning to the proof of Proposition \ref{propmaless}.
\end{proof}
To summarize the three cases above, we can state that:

\begin{prop}\label{propmorphism}
Let $M=\bigoplus_{j=0}^{n-1}\,^{\G_{d,n}^j}M_a$ be an $ii$-indecomposable representation. Then there exists an isomorphism $f:M\lra\,^{\G_{d,n}}M$ such that the induced automorphism $^{\G_{d,n}^{l-1}}f\circ\cdots\circ^{\G_{d,n}}f\circ f$ of $M$ is the identity, for some positive integer $l$. Moreover, if $n=m$ then $l=n$ and if $n\neq m$ then $l=d$. 
\end{prop}

In contrast, there remains an entire family of $ii$-indecomposable regular simple representations to explore, namely the  $M_{(j)}$ representations, defined by
\begin{enumerate}
    \item[($M_{j}$1)] For $1\leq i,\, i'\leq j$, the vector space associated with the vertex $i$ and $i'$ is $k_n$. For the remaining vertices, the vector space is $0$.
    \item[($M_{j}$2)] For $1\leq i\leq j$ and $2\leq i'\leq j$ the $k_n$-morphism associated to the arrows $\A_i$ and $\A_{i'}$ is $\id_{k_n}$. For the remaining arrows, the morphisms are the null morphism.
\end{enumerate}

And $N_{(j)}$ representations defined by
\begin{enumerate}
    \item[($N_{j}$1)] For $1\leq i\leq j$ and $2\leq i'\leq j+1$ the vector space associated to the vertex $i$ and $i'$ is $k_n$. For the remaining vertices, the vector space is $0$
    \item[($N_{j}$2)] For $2\leq i\leq j$ and $2\leq i'\leq j+1$ the $k_n$-morphism associated to the arrows $\A_i$ and $\A_{i'}$ is $\id_{k_n}$. For the remaining arrows, the morphisms are the null morphism.
\end{enumerate}
Figure \ref{figuremn}, illustrates a graphical representation of $M_{(j)}$ and $N_{(j)}$.
\begin{figure}[h]
\centering
\begin{tikzpicture}
[->,>=stealth',shorten >=1pt,auto,node distance=2cm,thick,main node/.style=]
  \node[main node] (1) {$k_n$};   
  \node[main node] (6) [below of =1]{$k_n$};   
  \node[main node] (2) [right of=1] {$\cdots$};
  \node[main node] (7) [below of =2]{$\cdots$};   
  \node[main node] (3) [right of=2] {$k_n$};
    \node[main node] (8) [below of =3]{$k_n$};
  \node[main node] (4) [right of=3] {$0$};
  \node[main node] (5) [below of=4] {$0$};
  \node[main node] (9) [right of=4] {$\cdots$};
   \node[main node] (10) [below of=9] {$\cdots$}; 
\node[main node] (11) [right of=9] {$0$};
   \node[main node] (12) [below of=11] {$0$};
   \path[every node/.style={font=\sffamily\small}]     
(1) edge node   {$1$} (6)         
     edge node         {$1$} (7)           
(2) edge node  {} (8)              
(3) edge node  {$1$} (8)
    edge node  {} (5)
(4) edge node  {} (5)              
   edge node         {} (10) 
(11) edge node {}(6)
    edge node {}(12)
   (9) edge node {}(12);
\end{tikzpicture}
\begin{tikzpicture}
[->,>=stealth',shorten >=1pt,auto,node distance=2cm,thick,main node/.style=]
  \node[main node] (1) {$k_n$};   
  \node[main node] (6) [below of =1]{$0$};   
  \node[main node] (2) [right of=1] {$\cdots$};
  \node[main node] (7) [below of =2]{$\cdots$};   
  \node[main node] (3) [right of=2] {$k_n$};
    \node[main node] (8) [below of =3]{$k_n$};
  \node[main node] (4) [right of=3] {$0$};
  \node[main node] (5) [below of=4] {$k_n$};
  \node[main node] (9) [right of=4] {$\cdots$};
   \node[main node] (10) [below of=9] {$\cdots$}; 
\node[main node] (11) [right of=9] {$0$};
   \node[main node] (12) [below of=11] {$0$};
   \path[every node/.style={font=\sffamily\small}]     
(1) edge node   {$0$} (6)         
     edge node         {$1$} (7)           
(2) edge node  {$1$} (8)              
(3) edge node  {$1$} (8)
    edge node  {} (5)
(4) edge node  {} (5)              
   edge node         {} (10) 
(11) edge node {}(6)
    edge node {}(12)
   (9) edge node {}(12);
\end{tikzpicture}
\caption{For $1\leq j\leq n$, we have the family of representations $M_{(j)}$ and $N_{(j)}$.} 
\label{figuremn}
\end{figure}
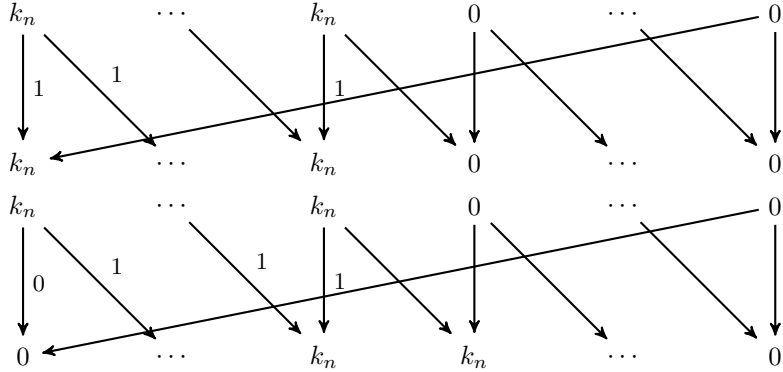
In general, we can think of the sum $\bigoplus_{i=0}^{n-1}\,^{\G_{n,n}^i}M_{(j)}$ as follows
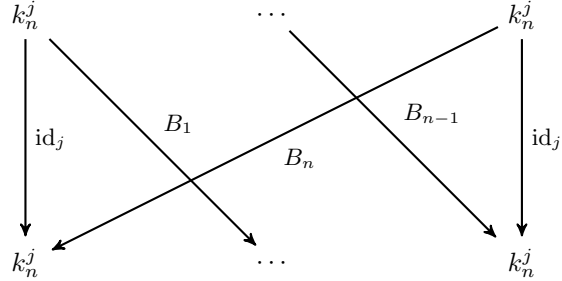
\begin{center}
\begin{tikzpicture}
[->,>=stealth',shorten >=1pt,auto,node distance=3.3cm,thick,main node/.style=]
  \node[main node] (1) {$k_n^j$};   
  \node[main node] (6) [below of =1]{$k_n^j$};   
  \node[main node] (2) [right of=1] {$\cdots$};
  \node[main node] (7) [below of =2]{$\cdots$};   
  \node[main node] (3) [right of=2] {$k_n^j$};
    \node[main node] (8) [below of =3]{$k_n^j$};
\path[every node/.style={font=\sffamily\small}]     
(1) edge node   {$\id_j$} (6)         
     edge node         {$B_1$} (7)           
(2) edge node  {$B_{n-1}$} (8)              
(3) edge node  {$\id_{j}$} (8) 
    edge node  {$B_{n}$} (6);
\end{tikzpicture}
\end{center}

where
  $$B_i=\left\{\begin{array}{cc}
         \id_j-E_{j+1-i,\,j+1-i}& \mbox{ si }1\leq i\leq j \\
         \left[\begin{array}{cc}
             0_{1\times j-1} &0_{1\times1}  \\
             \id_{j-1} &0_{j-1\times 1} 
         \end{array}\right]& \mbox{ si }j+1\leq i\leq n
    \end{array}\right.$$
and $E_{i,j}$ denotes the matrix whose entries are all $0$ except for the entry at position $i,j$, which is $1$.

\begin{remark}\label{rem:mor}
{\em Given any integer $n\geq 2$, the following equalities are satisfied:
     \begin{eqnarray*}
         E_j^{-1}B_i&=&B_{i+1}E_j^{-1}\mbox{ if }1\leq i\leq j-1\\
    \id_jB_j&=&B_{j+1}E_j^{-1}\\
    E_j^{-1}B_n&=&B_1\id_j.
     \end{eqnarray*}
Indeed, if $1\leq i\leq j-1$, we obtain that
    \begin{eqnarray*}
        E_j^{-1}B_i&=&E_j^{-1}(\id_j-E_{j+1-i,j+1-i})\\
        &=&E_j^{-1}-E_{j-i,j+1-i}\\
        &=&(\id_j-E_{j-i,j-i})E_j^{-1}\\
        &=&B_{i+1}E_j^{-1}.
    \end{eqnarray*}
Conversely, multiplying any matrix by $E_j^{-1}$from the left shifts each column one position to the right. Similarly, multiplying any matrix by $E_j^{-1}$ from the right shifts each row one position upward.}     
\end{remark}

Now, we aim to demonstrate that for these representations, $M_{(j)}$ and $N_{(j)}$, there exist morphisms as described in Proposition \ref{propmorphism}.  

\begin{lema}\label{lemmamj}
There exists a natural automorphism $f:\bigoplus_{i=0}^{n-1}\,^{\G_{n,n}^i}M_{(j)}\lra\,^{\G_{n,n}}(\bigoplus_{i=0}^{n-1}\,^{\G_{n,n}^i}M_{(j)})$, such that 
        $$\,^{\G_{n,n}^{n-1}}f\circ\cdots\circ\,^{\G_{n,n}}f\circ f=\id.$$
    \end{lema}
    \begin{proof}
        Let $f=((f_i,f_{i'}))_{1\leq i\leq n}$ be defined as follows
        $$
        f_{i'}=f_i=\left\{\begin{array}{cc}
            E_j^{-1}&\mbox{ if }1\leq i\leq j \\
             \id_j& \mbox{otherwise.} 
        \end{array}\right.
        $$
As a consequence of Remark \ref{rem:mor}, we have that $f$ is an automorphism from $\bigoplus_{i=0}^{n-1}\,^{\G_{n,n}^i}M_{(j)}$ to $\,^{\G_{n,n}}(\bigoplus_{i=0}^{n-1}\,^{\G_{n,n}^i}M_{(j)})$. Since $\,^{\G_{n,n}}f=((g_i,g_{i'}))_{1\leq i\leq n}$, where $g_i=f_{i+1}$, $g_{i'}=f_{(i+1)'}$, for $1\leq i\leq n-1$, and $g_n=f_{1}$, $g_{n'}=f_{(1)'}$, we conclude that 
$$
^{\G_{n,n}^{n-1}}f\circ\cdots\circ^{\G_{n,n}}f\circ f=(((E_j^{-1})^j,(E_j^{-1})^j))_{1\leq i\leq n}
$$
Since $(E_j^{-1})^j=\id_j$, this completes the proof of the proposition..
    \end{proof}
\begin{lema}\label{lemmanj}
There exists a natural automorphism $f:\bigoplus_{i=0}^{n-1}\,^{\G_{n,n}^i}N_{(j)}\lra\,^{\G_{n,n}}(\bigoplus_{i=0}^{n-1}\,^{\G_{n,n}^i}N_{(j)})$, such that 
        $$\,^{\G_{n,n}^{n-1}}f\circ\cdots\circ\,^{\G_{n,n}}f\circ f=\id.$$
    \end{lema}
    \begin{proof}
The proof follows a similar line of reasoning to the proof of Lemma \ref{lemmamj}.
    \end{proof}
The following remark establishes a crucial connection between the aforementioned representations and the string representations of the crown quiver. Specifically:
\begin{remark}\label{remarkstring}
{\em Let $n\in\Z_{>0}$ and $1\leq j\leq n$, then $M_{(j)}=M(s_{1',2j-1})$ and $N_{(j)}=M(s_{1,2j-3})$. Moreover, for $1\leq i\leq n-1$ we have $^{\G_{d,n}^i}M_{(j)}=M(s_{(n+1-i)',\,2j-1})$ and $^{\G_{d,n}^i}N_{(j)}=M(s_{n+1-i,\,2j-1})$. Indeed, it follows from the definition of string module and $\G_n$, see Subsections \ref{string} and \ref{ncrown}.}
\end{remark}

\begin{teo}\label{classifi}
   Let $N$ be a regular simple $k_d\mathcal{Q}_n$-module. Then $N$ satisfies the conditions of Theorem \ref{lemma3.7} if and only if $N$ is isomorphic to either $M_a$, for a generic element $a\in k_m$, or is isomorphic to $M_{(j)}$ or $N_{(j)}$, for some $1\leq j\leq n$. 
\end{teo}
\begin{proof}
If $N$ is isomorphic to either $M_a$, for a generic element $a\in k_m$, or is isomorphic to $M_{(j)}$ or $N_{(j)}$, for some $1\leq j\leq n$, then $N$ satisfies the conditions of Theorem \ref{lemma3.7}, due to Proposition \ref{propmorphism}, Lemmas \ref{lemmamj} and \ref{lemmanj}, respectively.

Conversely, suppose $N$ is a regular simple $k_d\mathcal{Q}_n$-module that satisfies the conditions of Theorem \ref{lemma3.7}. Invoking Theorem \ref{teostringband}, Lemmas \ref{lemmastring} and \ref{lemmabands} and Remark \ref{remarkstring}, we conclude that the indecomposable regular representations are isomorphic to $M_{(j)}$ or $N_{(j)}$ or $M(S',\varphi)$. If $N$ is not isomorphic to $M_{(j)}$ or $N_{(j)}$, then $N\cong M(S',\varphi)$.  Since $N$ is also simple, this implies that $(V,\varphi)=(k_d, a)$, for some $a\in k_d$. The only elements that satisfy the conditions of Theorem \ref{lemma3.7} are the generic elements. Therefore, $N\cong M_a$ for some generic $a\in k_d$.
\end{proof}

A regular simple representation $X$ is said to be \textit{homogeneous} if $k_d\otimes X\cong\bigoplus_{j=0}^{n-1}\,^{\G_{d,n}^j}M_a$, where $a\in k_n$ is a generic element. If $k_d\otimes X\cong\bigoplus_{j=0}^{n-1}\,^{\G_{n,n}^j}M_{(j)}$ or $k_d\otimes X\cong\bigoplus_{j=0}^{n-1}\,^{\G_{n,n}^j}N_{(j)}$, the representation $X$ is said to be {\em non-homogenous}.
\begin{remark}
{\em Theorem \ref{classifi} establishes that if $M$ is homogeneous, there exist at most $mn$ distinct generic elements in $k_m$ that give rise to the ``same'' simple regular representation. Furthermore, for $\Lambda_n$, the homogeneous simple regular representations are parametrized by $\spec(k_n[x])$.}
\end{remark}
We now seek to establish a suitable representative for these regular simple representations. When $X$ is non-homogenous we can construct a satisfactory representative, for all $1\leq j\leq n$, as demonstrated in Theorem \ref{teoclassmn}.

For homogenous $X$, we only have a partial classification when $k_d\otimes X\cong\bigoplus_{j=0}^{n-1}\,^{\G_{d,n}^j}M_{x^n}$, where $x^n$ is a generic element for $x\in k_n$ (see Theorem \ref{teoclassma}).

To accomplish this, for a positive integer $m$, we begin by defining the matrix $\A_{m}\in\M_{m\times m}(k)$ as follows:
$$
(\A_{m})_{ij}=\left\{\begin{array}{ccc}\delta_{i+1,\,j}&\mbox{si}&i<m,\\\delta_{1,\,j}\E&\mbox{si}&i=m,\end{array}\right.
$$
i.e, in a more explicit form
$$
\A_{m}=\left(\begin{array}{ccccc}0&1&0&\cdots&0\\ 0&0&1&\cdots&0\\\vdots&\vdots&\vdots&\ddots&\vdots\\0&0&0&\cdots&1\\\E&0&0&\cdots&0\end{array}\right).
$$
For the sake of completeness, we define $\A_{1}=\E$. The minimal polynomial of $\A_{m}$, can be verified without difficulty, is $y^m-\E\in k[y]$.

Now, from Lemma 2.7 in \cite{GR22}, we can state that:

 
\begin{prop}\label{prop3.10}
Consider $x$ a generic element in $k_m$. Then, $\bigoplus_{j=0}^{m-1}\,^{\G_{d,n}^j}M_{x}$ is isomorphic to 
\begin{center}
\begin{tikzpicture}
[->,>=stealth',shorten >=1pt,auto,node distance=3.3cm,thick,main node/.style=]
  \node[main node] (1) {$k_d^m$};   
  \node[main node] (6) [below of =1]{$k_d^m$};   
  \node[main node] (2) [right of=1] {$\cdots$};
  \node[main node] (7) [below of =2]{$\cdots$};   
  \node[main node] (3) [right of=2] {$k_d^m$};
    \node[main node] (8) [below of =3]{$k_d^m$.};
\path[every node/.style={font=\sffamily\small}]     
(1) edge node   {$\A(x)$} (6)         
     edge node         {$\id_m$} (7)           
(2) edge node  {$\id_m$} (8)              
(3) edge node  {$\id_m$} (8) 
    edge node  {$\id_m$} (6);
\end{tikzpicture}
\end{center}
\end{prop}
\begin{proof}
Suppose $x\in k_m$ is a generic element. Since $^{\G_{d,n}^j}M_{x}\cong M_{\Si_d^j(x)}$, we can consider $\bigoplus_{j=0}^{n-1}\,^{\G_{d,n}^j}M_{x}$ as follows:
\begin{center}
\begin{tikzpicture}
[->,>=stealth',shorten >=1pt,auto,node distance=3.3cm,thick,main node/.style=]
  \node[main node] (1) {$k_d^m$};   
  \node[main node] (6) [below of =1]{$k_d^m$};   
  \node[main node] (2) [right of=1] {$\cdots$};
  \node[main node] (7) [below of =2]{$\cdots$};   
  \node[main node] (3) [right of=2] {$k_d^m$};
    \node[main node] (8) [below of =3]{$k_d^m$};
\path[every node/.style={font=\sffamily\small}]     
(1) edge node   {$D_{x}$} (6)         
     edge node         {$\id_m$} (7)           
(2) edge node  {$\id_m$} (8)              
(3) edge node  {$\id_m$} (8) 
    edge node  {$\id_m$} (6);
\end{tikzpicture}
\end{center}
where $D_{x}=\diag({x},\Si_m({x}),\cdots,\,\Si_m^{m-1}({x}))$. By Lemma 2.7 in \cite{GR22}, it follows that there exist $A\in\GL_m(k_m)$ such that $\A(x) A=AD_{x }.$

\end{proof}

\begin{teo}\label{teoclassma}
Consider $x\in k_m$ such that $x^n$ is a generic element. Then, the representation $\varphi_x:k_n^m\oplus^{\Si_n}k_n^m\longrightarrow k_n$, whose matrix is $[\A(x)\id_m]$, is a regular simple representation.
\end{teo}
\begin{proof}
We know that $k_d\otimes\varphi_x$ can be identified with the representation
\begin{center}
\begin{tikzpicture}
[->,>=stealth',shorten >=1pt,auto,node distance=3.3cm,thick,main node/.style=]
  \node[main node] (1) {$k_d^m$};   
  \node[main node] (6) [below of =1]{$k_d^m$};   
  \node[main node] (2) [right of=1] {$k_d^m$};
  \node[main node] (7) [below of =2]{$k_d^m$};   
  \node[main node] (3) [right of=2] {$\cdots$};
    \node[main node] (8) [below of =3]{$\cdots$};
  \node[main node] (4) [right of=3] {$k_d^m$};
  \node[main node] (5) [below of=4] {$k_d^m$};
\path[every node/.style={font=\sffamily\small}]     
(1) edge node   {$\A(x)$} (6)         
     edge node         {$\Si_n^{n-1}(\id_m)$} (7)           
(2) edge node  {${\scriptscriptstyle\Si_n^{n-1}(\A(x))}$} (7)              
(4) edge node  {$\Si_n(\A(x))$} (5)              
   edge node         {$\id_m$} (6)
   (3) edge node  {$\Si_n(\id_m)$} (5) ;
\end{tikzpicture}
\end{center}
where $d=\mcm(n,m)$. Since $\Si_n^j(\id_m)=\id_m$ and $\Si_n^j(\A(x))=\A(x)$, we have that the corresponding representation is 
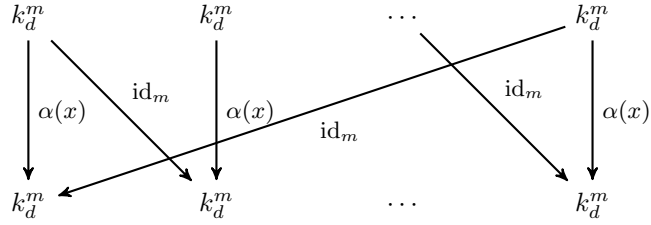
\begin{center}
\begin{tikzpicture}
[->,>=stealth',shorten >=1pt,auto,node distance=2.5cm,thick,main node/.style=]
  \node[main node] (1) {$k_d^m$};   
  \node[main node] (6) [below of =1]{$k_d^m$};   
  \node[main node] (2) [right of=1] {$k_d^m$};
  \node[main node] (7) [below of =2]{$k_d^m$};   
  \node[main node] (3) [right of=2] {$\cdots$};
    \node[main node] (8) [below of =3]{$\cdots$};
  \node[main node] (4) [right of=3] {$k_d^m$};
  \node[main node] (5) [below of=4] {$k_d^m$};
\path[every node/.style={font=\sffamily\small}]     
(1) edge node   {$\A(x)$} (6)         
     edge node         {$\id_m$} (7)           
(2) edge node  {$\A(x)$} (7)              
(4) edge node  {$\A(x)$} (5)              
   edge node         {$\id_m$} (6)
   (3) edge node  {$\id_m$} (5) ;
\end{tikzpicture}
\end{center}
Nevertheless, this representation is isomorphic to
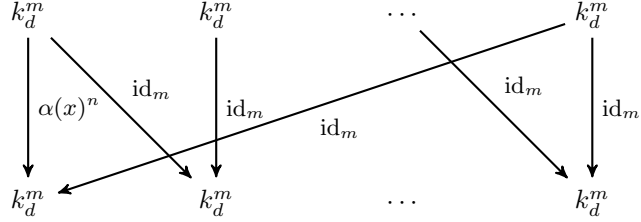
\begin{center}
\begin{tikzpicture}
[->,>=stealth',shorten >=1pt,auto,node distance=2.5cm,thick,main node/.style=]
  \node[main node] (1) {$k_d^m$};   
  \node[main node] (6) [below of =1]{$k_d^m$};   
  \node[main node] (2) [right of=1] {$k_d^m$};
  \node[main node] (7) [below of =2]{$k_d^m$};   
  \node[main node] (3) [right of=2] {$\cdots$};
    \node[main node] (8) [below of =3]{$\cdots$};
  \node[main node] (4) [right of=3] {$k_d^m$};
  \node[main node] (5) [below of=4] {$k_d^m$};
\path[every node/.style={font=\sffamily\small}]     
(1) edge node   {$\A(x)^n$} (6)         
     edge node         {$\id_m$} (7)           
(2) edge node  {$\id_m$} (7)              
(4) edge node  {$\id_m$} (5)              
   edge node         {$\id_m$} (6)
   (3) edge node  {$\id_m$} (5) ;
\end{tikzpicture}
\end{center}
where the isomorphism is given by $f=((f_i,f_{i'}))_{1\leq i\leq n-1}$, where $f_i=\A(x)^{n-i}$ and $f_{i'}=f_{j}$, $i\equiv j-1(\modu n)$. From the equation $\A(x^n)=\A(x)^n$, and due to Proposition \ref{prop3.10} and Theorem \ref{lemma3.7}, we can conclude that $\varphi_x$ is regular simple representation.
\end{proof}

\begin{lema}\label{lemma3.17}
Let $n\in \Z$ and $1\leq j\leq n$. Then, there exists an isomorphism between $\bigoplus_{i=0}^{n-1}\,^{\G_{n,n}^i}M_{(j)}$ and the representation 
        \begin{center}
\begin{tikzpicture}
[->,>=stealth',shorten >=1pt,auto,node distance=3.3cm,thick,main node/.style=]
  \node[main node] (1) {$k_n^j$};   
  \node[main node] (6) [below of =1]{$k_n^j$};   
  \node[main node] (2) [right of=1] {$\cdots$};
  \node[main node] (7) [below of =2]{$\cdots$};   
  \node[main node] (3) [right of=2] {$k_n^j$};
    \node[main node] (8) [below of =3]{$k_n^j$};
\path[every node/.style={font=\sffamily\small}]     
(1) edge node   {$\id_j$} (6)         
     edge node         {$B$} (7)           
(2) edge node  {$B$} (8)              
(3) edge node  {$\id_{j}$} (8) 
    edge node  {$B$} (6);
\end{tikzpicture}
\end{center}
where $B=\left[\begin{array}{cc}
             0_{1\times j-1} &0_{1\times1}  \\
             \id_{j-1} &0_{j-1\times 1} 
         \end{array}\right]$.
    \end{lema}
\begin{proof}
It is easy to verifies that 
$$
\left[\begin{array}{cc}
             0_{1\times j-1} &0_{1\times1}  \\
             \id_{j-1} &0_{j-1\times 1} 
         \end{array}\right](E_j)^{j+1-i}=(E_j)^{j-i}B_i\mbox{ si }1\leq i\leq j.
$$
Therefore, the isomorphism with the desired properties is given by $f((f_i,f_{i'}))_{1\leq i\leq n}$, which is defined as follows:

$$
f_{i'}=f_i=\left\{\begin{array}{cc}
            E_j^{j+1-i} &\mbox{ if }1\leq i\leq j,  \\
            \id_j. &\mbox{otherwise} 
        \end{array}\right.
$$
    \end{proof}

\begin{lema}\label{lemma3.18}
Let $n\in \Z$ and $1\leq j\leq n$. Then, there exists an isomorphism between $\bigoplus_{i=0}^{n-1}\,^{\G_{n,n}^i}N_{(j)}$ and the representation 
        \begin{center}
\begin{tikzpicture}
[->,>=stealth',shorten >=1pt,auto,node distance=3.3cm,thick,main node/.style=]
  \node[main node] (1) {$k_n^j$};   
  \node[main node] (6) [below of =1]{$k_n^j$};   
  \node[main node] (2) [right of=1] {$\cdots$};
  \node[main node] (7) [below of =2]{$\cdots$};   
  \node[main node] (3) [right of=2] {$k_n^j$};
    \node[main node] (8) [below of =3]{$k_n^j$};
\path[every node/.style={font=\sffamily\small}]     
(1) edge node   {$B$} (6)         
     edge node         {$\id_j$} (7)           
(2) edge node  {$\id_j$} (8)              
(3) edge node  {$B$} (8) 
    edge node  {$\id_j$} (6);
\end{tikzpicture}
\end{center}
where $B=\left[\begin{array}{cc}
             0_{1\times j-1} &0_{1\times1}  \\
             \id_{j-1} &0_{j-1\times 1} 
         \end{array}\right]$.
    \end{lema}

\begin{teo}\label{teoclassmn}
Consider $n\in \Z$ be positive and $1\leq j\leq n$. Then, the representation $\varphi:k_n^j\oplus\,^{\Si_n}k_n^j\longrightarrow k_n^j$, whose matrix is $[B\id_j]$ or $[\id_jB]$, is a regular simple representation, where $B$ is the matrix described in Lemmas \ref{lemma3.17} and \ref{lemma3.18}.
\end{teo}
\begin{proof}
Consider the representation $\varphi:k_n^j\oplus\,^{\Si_n}k_n^j\longrightarrow k_n^j$, whose matrix is either $[A\id_j]$ or $[\id_jB]$. We will analyze the first case, namely the matrix $[B\id_j]$, since the second case is completely analogous. As established in Section \ref{Invariantrep}, we know that $k_n\otimes \varphi$ has the following description:
    
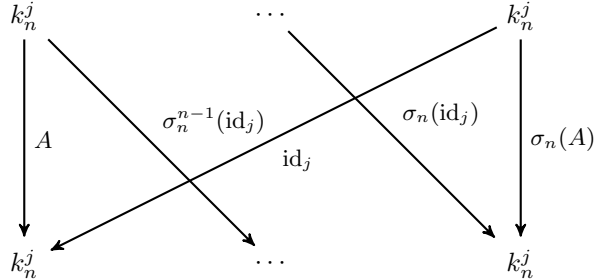
\begin{center}
     \begin{tikzpicture}
[->,>=stealth',shorten >=1pt,auto,node distance=3.3cm,thick,main node/.style=]
  \node[main node] (1) {$k_n^j$};   
  \node[main node] (6) [below of =1]{$k_n^j$};   
  \node[main node] (2) [right of=1] {$\cdots$};
  \node[main node] (7) [below of =2]{$\cdots$};   
  \node[main node] (3) [right of=2] {$k_n^j$};
    \node[main node] (8) [below of =3]{$k_n^j$};
\path[every node/.style={font=\sffamily\small}]     
(1) edge node   {$A$} (6)         
     edge node         {$\Si_n^{n-1}(\id_j)$} (7)           
(2) edge node  {$\Si_n(\id_j)$} (8)              
(3) edge node  {$\Si_n(A)$} (8) 
    edge node  {$\id_j$} (6);
\end{tikzpicture}
\end{center}
where $\Si_n^i(A)=A$ and $\Si_n^i(\id_j)=\id_j$, for any $1\leq i\leq n-1$. By Lemma \ref{lemma3.18}, we have that $k_n\otimes \varphi$ is isomorphic to $\bigoplus_{i=0}^{n-1}\,^{\G_{n,n}^i}N_{(j)}$ and, by Theorem \ref{lemma3.7}, we conclude that $k_n\otimes \varphi$ is a regular simple representation.
\end{proof}

\subsection{The classification on $\Lambda_1$}\label{Seclambda}
In this case, we have that $\Lambda_1=\left[\begin{array}{cc}
    k & k^2 \\
    0 &k 
\end{array}\right]$, which is isomorphic to $k\mathcal{Q}_1$ the path algebra of 2-Kronecker quiver:
$$
\mathcal{Q}_1=\begin{tikzcd}
     1
& 2 \arrow[l, shift left]
        \arrow[l, shift right]
\end{tikzcd}.
$$


The classification of all irreducible representations of the Kronecker quiver over a field is a well-established problem, as evidenced in \cite{Ri13} for $n-$Kronecker quiver and \cite{Hu16} for the specific case 2-Kronecker quiver. The regular simple representations of $k_n\mathcal{Q}_1$ are
$$
M_a\coloneqq\begin{tikzcd}
     k_n
& k_n \arrow[l, shift left,"1"]
        \arrow[l, shift right,"a"']
\end{tikzcd}\,\,\,\,\,\,\,\,\,\,M_{(1)}\coloneqq\begin{tikzcd}
     k_n
& k_n \arrow[l, shift left,"0"]
        \arrow[l, shift right,"1"']
\end{tikzcd}.
$$
As demonstrated in \cite{Hu16} there exists a bijection between the regular semisimple modules and the irreducible polynomials over the algebra $k[x]$. Furthermore, for each irreducible monic polynomial  $p(x)\in k[x]$, with $\deg(p)=n$, there corresponds the following representation:
$$
\begin{tikzcd}
     k_n^n
& k_n \arrow[l, shift left,"\id_n"]
        \arrow[l, shift right,"C(p)"']
\end{tikzcd},
$$
where $C(p)$ is the companion matrix of $p(x)$. Recall that if $a\in k_n$ is a root of $p(x)$ then $a$  is generic, i.e $\Si_n^j(a)\neq a$ for $1\leq j\leq n-1$. Thus $C(p)=V^{-1}\diag(a,\Si_n(a),\cdots,\Si_n^{n-1}(a))$V, where $V$ is the Vandermonde matrix:
$$
V=\left[\begin{array}{ccccc}1&a&a^2&\cdots&a^{n-1}\\1&\Si_n(a)&(\Si_n(a))^2&\cdots&(\Si_n(a))^{n-1}\\
  \vdots&\vdots&\vdots&\ddots&\vdots\\1&\Si_n^{n-1}(a)&(\Si_n^{n-1}(a))^2&\cdots&(\Si_n^{n-1}(a))^{n-1}\end{array}\right].
  $$
Therefore, the representations of $p(x)$ is equivalent to
$$ 
\begin{tikzcd}
     k_n^n
& k_n \arrow[l, shift left,"\id_n"]
        \arrow[l, shift right,"D_a"']
\end{tikzcd},
$$
where $D_a=\diag(a,\Si_n(a),\cdots,\Si_n^{n-1}(a))$. 

Alternatively, a representation that is not part of this family is 
$$
\begin{tikzcd}
     k
& k\arrow[l, shift left,"0"]
        \arrow[l, shift right,"1"']
\end{tikzcd}\oplus \begin{tikzcd}
     k
& k \arrow[l, shift left,"1"]
        \arrow[l, shift right,"0"']
\end{tikzcd}.
$$
\begin{prop}
Let $n$ a positive integer and $M$ a $\Lambda_1-$module such that either $M=\bigoplus_{j=0}^{n-1}\,^{\G_n^j}M_a$ or $M=M_{(1)}\oplus\,^{\G_2}M_{(1)}$. Then, there exists a morphism $f:M\longrightarrow\,^{\G_n}M$ such that the induced automorphism $^{\G_{n}^{n-1}}f\circ\cdots\circ^{\G_{n}}f\circ f$ of $M$ is the identity.
\end{prop}
\begin{proof}
It suffices to take $f=(E_n,E_n)$. 
\end{proof}
It is straightforward to adapt Lemma \ref{obs6} to the case of $\mathcal{Q}_1$. Therefore, we can conclude the following proposition.
\begin{prop}
Each regular simple representation of $\Lambda_1$ is isomorphic to either $\bigoplus_{j=0}^{n-1}\,^{\G_n^j}M_a$, for some generic element $a\in k_n$, or $M_{(1)}\oplus\,^{\G_2}M_{(1)}$.
\end{prop}
Finally, for $\Lambda_1$, we can effectively describe the structure of a regular simple representation.
\begin{teo}
Let $a\in k_n$ be a generic element. Then, the representation $\varphi_a:k^n\oplus k^n\lra k^n$, whose matrix is given by $[\A(a)\id_n]$, is a regular simple representation.
\end{teo}
\begin{proof}
    Since $a$ is generic it suffices to show that $k_n\otimes_k \varphi_x$ is isomorphic to
      $$ 
      \begin{tikzcd}
     k_n^n
& k_n \arrow[l, shift left,"\id_n"]
        \arrow[l, shift right,"D_a"']
\end{tikzcd},
$$
where $D_a=\diag(x,\Si_n(x),\cdots,\Si_n^{n-1}(x))$. Thus, the conclusion is a consequence of Lemma 2.7 in \cite{GR22}.
\end{proof}
\section*{Acknowledgement}
The second author gratefully acknowledges the support provided by CODI (Universidad de Antioquia, UdeA) through its Postdoc program. The first and third authors were partially supported by CODI (Universidad de Antioquia, UdeA) by project numbers 2022-52654 and 2023-62291, respectively.
 \addcontentsline{toc}{section}{References} 

\end{document}